\documentclass[preprint, 11pt]{amsart}
\usepackage{amsmath}
\usepackage[margin=1in]{geometry}
\usepackage[utf8]{inputenc}
\usepackage{amsthm}
\usepackage{amssymb}
\usepackage{enumitem} 
\usepackage{bbm}
\usepackage{amsmath}
\usepackage{csquotes}
\usepackage{blindtext}
\usepackage{mathtools}
\usepackage{breqn}
\usepackage{xfrac}
\usepackage{pdfpages}
\usepackage{graphicx}
\graphicspath{ {images/} }
\usepackage{dirtytalk}
\usepackage{tikz}
\usepackage{mathrsfs}
\usepackage{pgfplots}
\pgfplotsset{compat=1.14}
\usetikzlibrary{patterns}
\usepackage{float}

\DeclareMathOperator\supp{supp}

\DeclarePairedDelimiter\ceil{\lceil}{\rceil}

\theoremstyle{plain}
\newtheorem{theorem}{Theorem}[section]

\newtheorem*{theorem*a}{Theorem A}
\newtheorem*{theorem*b}{Theorem B}
\newtheorem*{theorem*c}{Theorem C}

\theoremstyle{plain}
\newtheorem{corollary}{Corollary}[theorem]

\theoremstyle{definition}
\newtheorem{example}{Example}[section]

\theoremstyle{remark}
\newtheorem{remark}{Remark}[section]

\theoremstyle{plain}
\newtheorem{lemma}{Lemma}[section]

\theoremstyle{plain}

\theoremstyle{definition}
\newtheorem{definition}{Definition}[section]

\theoremstyle{plain}

\newcommand{\ti}[1]{\tilde{#1}}
\newcommand{\ul}[1]{\underline{#1}}
\newcommand{\ov}[1]{\overline{#1}}
\newcommand{\R}{\mathbbm{R}}
\newcommand{\C}{\mathbbm{C}}

\newcommand{\e}{\varepsilon}
\newcommand{\Om}{\Omega}
\newcommand{\N}{\mathbb{N}}
\newcommand{\Z}{\mathbbm{Z}}
\newcommand{\mi}{\mathrm{i}}
\newcommand{\euler}{e}
\title{Stable Intersections of Conformal Cantor Sets}
\author{Hugo Araújo and Carlos Gustavo Moreira}

\begin{document}

\maketitle

\begin{small}
ABSTRACT. We investigate stable intersections of conformal Cantor sets and their consequences to dynamical systems. First we define this type of Cantor set and relate it to horseshoes appearing in automorphisms of $\C^2$.  Then we study limit geometries, objects related to the asymptotic shape of the Cantor sets, to obtain a criterion that guarantees stable intersection between some configurations. Finally we show that the Buzzard construction of a Newhouse region on $Aut(\C^2)$ can be seem as a case of stable intersection of Cantor sets in our sense and give some (not optimal) estimative on how \say{thick} those sets have to be.

\end{small}

\section{Introduction}
The theory of regular Cantor sets in the real line has played a central role in the study of dynamical systems, specially in relation to their hiperbolicity\footnote{that is, the existence of a decomposition of the tangent bundle over its limit set into two subbundles, one uniformly contracted by the action of the tangent map the other one uniformly expanded.}. One of its first results were the works of Newhouse \cite{n_1}, \cite{n_2} and \cite{n_3} where he showed that there is an open set $U$ in the space of $C^2$ diffeomorphisms of a compact surface ($Diff^2(\mathcal{M}^2)$) such that any diffeomorphism in $U$ is not hyperbolic. More than that, he observed that diffeomorphisms exhibthing a homoclinic tangency belonged to the closure of the set $U$. 

In those works he associated the presence of a tangency between the stable and unstable manifolds of a horseshoe, that is, a homoclinic tangency, to an intersection between two Cantor sets (constructed from the dynamical system). Then, an open set in $Diff^2(\mathcal{M}^2)$ with persistence of homoclinic tangencies is constructed via a pair of Cantor sets $(K_1, K_2)$ that have stable intersections, that is, $\tilde{K}_1\cap\tilde{K}_2$ is non-empty for any small pertubations $\tilde{K}_1$ and $\tilde{K}_2$  of $K_1$ and $K_2 $. To construct such a pair he developed a sufficient criterion for this phenomenon: the \emph{gap lemma}. Precisely, he defined $\tau(K)$, the thickness of a Cantor set $K$, a positive real number associated to the geometry of the gaps of $K$ and showed that if  the product $\tau(K_1) \cdot \tau(K_2)$ is larger than one for a pair of Cantor sets $K_1$ and $K_2$, then the pair $(K_1,K_2+t)$ has stable intersection for certain values of $t$. 

Similar results were obtained in other contexts, such as the work of Palis and Viana on larger dimensions \cite{palisviana} and of Duarte \cite{duarte_2000} on conservative systems. We are more interested however in the work of Buzzard \cite{buzz}, where he found an open region in the space of automorphisms of $\C^2$, i.e., holomorphic diffeomorphisms with holomorphic inverse, with persistent homoclinic tangencies. His strategy was very similar to the first work of Newhouse \cite{n_1}, constructing a \say{very thick} horseshoe, such that the Cantor sets, this time living in the complex plane, associated to it would also be \say{very thick}. However, the concept of thickness does not have a simple extension to this complex setting and so the argument to guarantee intersections between the Cantor sets after a small pertubation is different. It is worth noticing that a version of the \emph{gap lemma} for holomorphic Cantor sets was recently discovered, in 2018, by Biebler (see \cite{biebler}).

The objective of this paper is to present a criterion for stable intersection of Cantor sets that works for the Cantor sets derived from horseshoes appearing in automorphisms of $\C^2$. We begin by defining conformal Cantor sets. These sets are, roughly speaking,  the maximal invariant set of a $C^{1+\e}$ expansive map $g$ defined on a subset of $\R^2$ (satisfying some properties, see section 2.1 for details) with the key hypothesis that its derivative is conformal over the invariant set, that is, the Cantor set itself. Throughout this article we will freely identify a conformal operator over $\R^2$ with the operator over $\C$ given by a multiplication by a complex number. Precisely, if $A=\begin{bmatrix}\alpha &-\beta \\ \beta &\alpha\end{bmatrix}, \alpha, \beta \in \R$, we identify it with the operator over $\C$ given by the multiplication by $\alpha + \beta \cdot\mi$. Also, a map between two open sets of $\mathbb{C}^n$ is said to be $C^r\text{, } r \in \mathbbm{R}$, when seen as a map between two open subsets of $\mathbb{R}^{2n}$ it is $C^r$.

We then proceed to show that this is the appropriate concept to study horseshoes appearing in automorphisms of $\C^2$, which we call complex horseshoes. It is necessary to observe that this nomenclature already appears in the literature and was introduced in the thesis of Oberste-Vorth \cite{ob} as a complex version of the \emph{Smale horseshoe}. These complex horseshoes are a particular case of the concept of horseshoe presented in this text. Also, the work of Oberste-Vorth shows the existence of complex horseshoes whenever there is a transversal homoclinic intersection in an automorphism of $\C^2$, a fact that justify our interest in this kind of objects. The details regarding these objects are given in the second section of this paper, and its main theorem is copied below:

\begin{theorem*a} Let $\Lambda$ be a complex horseshoe for a automorphism $G \in Aut(\C^2)$ and $p$ be a periodic point $\Lambda$. Then, if $\varepsilon$ is sufficiently small, there is a parametrization  $\pi: U \subset \C \to W^u_\varepsilon(p)$ such that $\pi^{-1}(W^u_\varepsilon(p) \cap \Lambda)$ is a conformal Cantor set in the complex plane.
\end{theorem*a} 

In the third section we extend the \emph{recurrent compact criterion} created by Moreira and Yoccoz in \cite{my} to this type of Cantor set. Here we see the importance of the conformality of our sets. It allows us to construct \emph{limit geometries}, that are roughly speaking approximation of the asymptotic shape of small pieces of the Cantor sets. The set of all limit geometries (for a given Cantor set) is a compact set. Because of that we can prove the recurrent compact criterion: if, for some pair of Cantor sets, we can find a compact set of pairs of affine configurations of limit geometries that is carried to its own interior by renormalization operators, that is, a \emph{recurrent compact set}, then the original pair of Cantor sets, after an affine transformation on one of its entries, has stable intersection. The concepts in the previous statement are defined and explained in section 3.

The concept of stable intersection is very similar to the one in the real line setting. We say that two Cantor sets $K$ and $K'$ are close to each other when the maps defining them are close to each other and so are the connected pieces of their domains $G(a)$ and $G'(a)$ (see the second section for details). Also, we define a configuration of a piece $G(a)$ as a $C^{1+\varepsilon}$ embedding $h: G(a) \to \C$ (see section \ref{attr} for details). That way, given  a  pair of configurations ($h, \,h'$), also referred to as a configuration pair or a relative positioning, we say it has stable intersections whenever, for any pair $(\tilde{h},\tilde{h}')$ close to ($h, \,h'$) and any pair of Cantor sets $\tilde{K},\tilde{K}'$ on a small neighbourhood of $(K,\,K')$ the intersection between $\tilde{h}(\tilde{K})$ and $\tilde{h}'(\tilde{K}')$ is non-empty (see section \ref{recc} for the details). The main result of this section is the following.
\begin{theorem*b} The following properties are true:
\begin{enumerate}
    \item Every recurrent compact set is contained on an immediately recurrent compact set.
    \item Given a recurrent compact set $\mathcal{L}$ (resp. immediately recurrent) for $g$, $g'$, for any $\ti{g}$,$\ti{g}'$ in a small neighbourhood of  $(g,g') \in \Om_{\Sigma} \times \Om_{{\Sigma}'}$ we can choose points $\ti{c}_a \in \ti{G}(a) \subset \ti{K}$ and  $\ti{c}_{a'} \in \ti{G}(a') \subset \ti{K}'$ respectively close to the pre-fixed $c_a$ and $c_{a'}$ in a manner that $\mathcal{L}$ is also a recurrent compact set for $\ti{g}$ and $\ti{g}'$.
    \item Any relative configuration contained in a recurrent compact set has stable intersections.
\end{enumerate}
\end{theorem*b} 
It is important to observe that the work of Moreira and Yoccoz \cite{my} was done to solve a conjecture of Palis: for generic pairs of Cantor sets in the real line $K_1$ and $K_2$, $K_1-K_2$ contains an interval or has Lebesgue measure zero. This conjecture was inspired by the work of Palis and Takens \cite{pt}, where they proved a theorem that assured full density of hiperbolicity on a parameter family that generically unfolds a homoclinic tangency, provided that the Hausdorff dimension of the horseshoe is less than one. The recurrent compact criterion was one of the tools used by them to show that for generic pairs $K_1, \, K_2$ of Cantor sets whose sum of Hausdoff dimension is larger than one\footnote{if the sum of Hausdorff dimensions is less than one the set difference $K_2-K_1$ has Hausdorff dimension less than one and so Lebesgue measure zero.} there is a real number $t$ such that $K_1$ and $K_2 + t$ have a stable intersection, which implies in particular that $K_2-K_1$ contains an interval around $t$.

Another motivation is the work of Dujardin and Lyubich \cite{Dujardin2015}, in which they show that homoclinic tangencies are the main obstruction to $J^{*}$- stability, a concept related to the absence of bifurcation on the type of periodic points (saddle, node, repeler or indiferent). Results regarding families unfolding homoclinic tangencies are also possible with our techiniques and will appear in another paper.  The general dichotomy that the set difference of conformal Cantor sets on the complex plane generically have zero measure or contain an open set is under development in a joint work with Zamudio, who has developed the \emph{scale recurrence lemma} \cite{zam}, another important tool.

We end this paper showing that Buzzard's construction can be interpreted as a case of stable intersection of conformal Cantor sets derived from the recurrent compact criterion. We also give (non-optimal) estimatives on how \say{thick} the Cantor sets have to be. The main result of it is the following.

\begin{theorem*c}
There is $\delta$ sufficiently small for which the pair of Cantor sets $(K,K)$ defined for Buzzard's example has a recurrent compact set of affine configurations of limit geometries $\mathcal{L}$ such that $[\text{Id},\text{Id}] \in \mathcal{L}$.
\end{theorem*c}

We did not aim for an optimal estimative on $\delta$ because that would complicate the argument and it may be better to work with other constructions. It may also be possible to use the recurrent compact criterion to construct other families of Cantor set (considering the Buzzard's example as a family parametrized by $\delta$, the space between the pieces) that would have stable intersection with any other sufficiently general Cantor set. That could be useful when tackling the question whether automorphsisms displaying a homoclinic tangency lie in the closure of the open set of persistent tangencies, as showed by Newhouse in \cite{n_3} for the real case.  
\section{Dynamically defined conformal Cantor sets in the complex plane and their relation to horseshoes.}

In this section we define conformal (or equivalently asymptotically holomorphic) Cantor sets and establish some basic properties as well as its relation to complex horseshoes, which are important invariant hyperbolic sets of automorphisms of $\C^2$.

\subsection{Dynamically defined Conformal Cantor sets.}\label{first} 
A dynamically defined Cantor set in $\mathbbm{C}$ is given by the following data:

\begin{itemize}

\item A finite set $\mathbb{A}$ of letters and a set $B \subset \mathbb{A} \times \mathbb{A} $ of admissible pairs.

\item For each $a\in \mathbb{A}$ a compact connected set $G(a)\subset \mathbbm{C}$.
 
\item A $C^{1+\varepsilon}$ map $g: V \to \mathbbm{C}$ defined in an open neighbourhood $V$ of $\bigsqcup_{a\in \mathbb{A}}G(a)$.
\end{itemize}

This data must verify the following assumptions:

\begin{itemize}
    
\item The sets $G(a)$, $a \in \mathbb{A}$ are pairwise disjoint

\item $(a,b)\in B $ implies $G(b) \subset g(G(a))$, otherwise $G(b) \cap g(G(a)) = \emptyset$.

\item For each $a \in \mathbb{A}$ the restriction $g|_{G(a)}$ can be extended to a $C^{1+\varepsilon}$ embedding (with  $C^{1+\varepsilon}$ inverse) from an open neighborhood of $G(a)$ onto its image such that $m(Dg)>1$, where $m(A) := \displaystyle{\inf_{v\in \mathbbm{R}^2 \equiv \mathbbm{C}}\frac{||Av||}{||v||}}$, $A$ being a linear operator on $\mathbbm{R}^2$.

\item The subshift $(\Sigma, \sigma)$ induced by $B$, called the type of the Cantor set 

$$ \Sigma=\{ \ul{a}= (a_0, a_1, a_2, \dots  ) \in \mathbb{A}^{\mathbbm{N}}:(a_i,a_{i+1}) \in B, \forall i \geq 0\} \text{, }$$

$\sigma (a_0,a_1,a_2, \dots) = (a_1,a_2,a_3, \dots)$ is topologically mixing.

\end{itemize}

Once we have all this data we can define a Cantor set (i.e. a totally disconnected, perfect compact set) on the complex plane: 

\[ K=\bigcap_{n \geq 0} g^{-n}\left(  \bigsqcup_{a \in \mathbb{A}} G(a) \right) \] 

We will usually write only $K$ to represent all the data that defines a particular dynamically defined Cantor set. Of course, the compact set $K$ can be described in multiple ways as a Cantor set constructed with the objects above, but whenever we say that $K$ is a Cantor set we assume that one particular set of data as above is fixed. In this spirit, we may represent the Cantor set $K$ by the map $g$ that defines it as described above, since all the data can be inferred if we know $g$. Also, when we are working with two Cantor sets $K \text{ and }K'$ we denote all the defining data related to the second accordingly. In other words, $K'$ is given by a finite set $\mathbb{A}'$, a set $B'$ of admissible pairs, a function $g'$ defined on a neighbourhood of compact connected sets $G(a')$, etc. We use the same convention for future objects that will be defined related to Cantor sets, such as limit geometries and configurations.

\begin{definition}\label{maindef}(Conformal regular Cantor set)
We say that a regular Cantor set is \textit{conformal} whenever the map $g$ is conformal at the Cantor set $K$, that is, $\forall x \in K$, $Dg(x):\mathbbm{C} \equiv \mathbbm{R}^2 \to \mathbbm{C} \equiv \mathbbm{R}^2$ is a linear transformation that preserves angles or, equivalently, a multiplication by a complex number.

\end{definition}

There is a natural topological conjugation between the dynamical systems $(K,g|_K)$ and $(\Sigma, \sigma)$, the subshift $\Sigma$ induced by $B$. It is given by a homeomorphism $h: K \to  \Sigma$ that carries each point $x \in K$ to the sequence $\{a_n\}_{n \geq 0}$ that satisfies $g^n(x) \in G(a_n)$.

Associated to a Cantor set $K$ we define the sets 

$$\Sigma^{fin}=\{(a_0, \dots ,a_n): (a_i,a_{i+1}) \in B \: \forall i , 0 \leq i < n \},$$
$$\Sigma^-=\{(\dots, a_{-n}, a_{-n+1},\dots,a_{-1},a_0): (a_{i-1},a_i) \in B \: \forall i \leq 0\}. $$

Given $\ul{a}=(a_0, \dots, a_n)$, $\ul{b}=(b_0, \dots , b_m)$, $\ul{\theta}^1=(\dots,\theta^1_{-2},\theta^1_{-1},\theta^1_{0})$ and $\ul{\theta}^2=(\dots,\theta^2_{-2},\theta^2_{-1},\theta^2_{0})$ we denote:

\begin{itemize}
\item if $a_n=b_0$, $\ul{ab}=(a_0, \dots,a_{n}, b_1, \dots, b_m)$
\item if $\theta^1_0=a_0$, $\ul{\theta^1a} = (\dots, \theta^1_{-2},\theta^1_{-1}, a_0, \dots, a_n )$
\item if $\ul{\theta}^1 \neq \ul{\theta}^2$ and $\theta ^1 _0 = \theta ^2_ 0$,   $\ul{\theta}^1 \wedge \ul{\theta}^2=(\theta_{-j}, \theta_{-j+1}, \dots , \theta_0)$, in which $\theta _{-i} =\theta^1 _{-i}= \theta^2 _{-i} $ for all $i=0, \dots, j $ and $\theta^1_{-j-1} \neq \theta^2_{-j-1}$.
\end{itemize}

For $\ul{a}=(a_0, a_1, \dots , a_n) \in \Sigma^{fin}$ we say that it has size $n$ and define:
$$G(\ul{a})= \{x \in \bigsqcup_{a \in \mathbb{A}} G(a) ,  g^j(x) \in G(a_j), j=0,1,\dots, n \}$$
and the function $f_{\ul{a}}: G(a_n) \to G(\ul{a})$ by:

$$ f_{\ul{a}} =  g|^{-1}_{G(a_0)} \circ g|^{-1}_{G(a_1)} \circ \dots \circ (g|^{-1}_{G(a_{n-1})})|_{G(a_n)} . $$

Notice that $f_{(a_i,a_{i+1})} = g|^{-1}_{G(a_{i})}$. 


In our definition we did not require the pieces $G(a)$ to have non-empty interior. However, if this is not the case, it is easy to see that we can choose $\delta$ sufficiently small such that the sets $G^*(a)= V_\delta(G(a))$ satisfy:

\begin{enumerate}[label=(\roman*)]
 
 \item  $G^*(a)$ is open and connected.
 
 \item $G(a) \subset G^*(a)$ and $g|_{G(a)}$ can be extended to an open neighbourhood of $\ov{G^*(a)}$, such that it is a $C^{1+\varepsilon}$ embedding (with  $C^{1+\varepsilon}$ inverse ) from this neighbourhood to its image and $m(Dg)>\mu$.
 
 \item The sets $\ov{G^*(a)}$, $a \in \mathbb{A}$ are pairwise disjoint
 
 \item $(a,b) \in B$ implies $\ov{G^*(b)}\subset g(G^*(a))$, and $(a,b) \notin B$ implies  $\ov{G^*(b)}\cap \ov{g(G^*(a))} = \emptyset$

\end{enumerate}

With this notation we have the following lemma.

\begin{lemma} \label{size}let K be a dynamically defined Cantor set $K$ and $G^*(a)$ the sets from lemma 1.  Let $G^*(\ul{a})$ be defined is the same way as $G(\ul{a})$. There exist constants $C$ and $\mu >1$ such that:

$$diam(G^*(\ul{a})) < C\mu^{-n}$$

\end{lemma}

\begin{proof} The proof is essentially the same as in \cite{zam}. Let $\mu > 1$ such that $m(Dg)> \mu$ in $\sqcup_{a \in \mathrm{A}}G^*(a)$. For $\ul{a}\in \Sigma^{fin}$ let $d_{\ul{a}}$ be the metric

$$d_{\ul{a}(x,y)} = \inf_{\ul{\alpha}} l(\alpha)$$

where $\alpha$ runs through all smooth curves inside $G^*(\ul{a})$ that connect $x$ to $y$ and $l(\alpha)$ denotes the lengths of such curves. Since $g$ sends $G^*(a_0,a_1,\dots,a_n)$ diffeomorphically onto $G^*(a_1,\dots,a_n)$ and $m(Dg) > \mu$ then

$$ d_{(a_1, \dots, a_n)}(g(x),g(y)) \geq \mu \cdot d_{a_0, \dots, a_n}(x,y) .$$

for all $x,y \in G^* (a_0, \dots a_n)$. Therefore,

$$\text{diam}_{(a_0, \dots a_n)}(G^* (a_0, \dots a_n)) \leq \mu ^{-1} \cdot \text{diam}_{(a_1, \dots a_n)}(G^* (a_1, \dots a_n)) $$

where $\text{diam}_{\ul{a}}$ is the diameter with respect to $d_{\ul{a}}$. We conclude that, by induction, 

$$ \text{diam}(G^*(\ul{a}))\leq \text{diam}_{\ul{a}}(G^*(\ul{a})) \leq \mu^{-n} \cdot \text{diam}_{{a}_n}(G^*({a}_n)) .$$

Taking any $C$ larger than $\max_{a \in \mathrm{A}} \text{diam}(G^*(a))$ yields the result.
\end{proof}

As a consequence of this lemma we can see that 
\[K=\bigcap_{n \geq 0}g^{-n}\left(\bigsqcup_{a \in \mathbb{A}}G^*(a)\right)\]
since $G(\ul{a})\subset G^*(\ul{a})$ and $\text{diam}(G^*(\ul{a})) \rightarrow 0$.

In this manner, the sets $G(a)$ can be substituted by the sets $\ov{G^*(a)}$ in the definition of $K$. So in what follows, additionally to the properties in the definition of Cantor sets we suppose that $G(a)=\ov{\mathring{(G(a))}}$ and that $g$ can always be extended to a neighbourhood $V_a$ of $G(a)$ such that it is a $C^{1+\varepsilon}$ embedding (with $C^{1+\varepsilon}$ inverse) and $m(Dg) > \mu$ over $V_a$, which by lemma 2 implies that $\text{diam}(G(\ul{a}))< C \mu ^{-n}$, if $\ul{a} = (a_0, \dots a_n)$. The most important examples of conformal Cantor sets come from  intersections between compact parts of stable and unstable manifolds of periodic points and basic sets of saddle type of a automorphism of $\mathbbm{C}^2$ and, as we will see, we can construct them from sets $G(a)$ with these properties already.

Finally, we have the definition:

\theoremstyle{definition}
\begin{definition}{(The space $\Om_{\Sigma}$)} The set of all conformal regular Cantor sets $K$ with the type $\Sigma$ is defined as the set of all conformal Cantor sets described as above whose set of data includes an alphabet $\mathbb{A}$ and the set  $B$ of admissible pairs used in the construction of $\Sigma$. We denote it by $\Om_{\Sigma}$.

\end{definition} 

 This space will be seen as a topological space. The topology is generated by a basis of neighbourhoods $U_{K,\delta} \subset \Om_{\Sigma}, \, K \in \Om_{\Sigma}, \, \delta > 0  $,  the $U_{K,\delta}$ being the set of all conformal regular Cantor sets $K'$ given by $g': V' \to \C, \, V' \supset \bigsqcup_{a \in \mathbb{A}} G'(a)$ such that $G(a) \subset V_{\delta}(G'(a))$, $G'(a) \subset V_{\delta}(G(a))$ (that is, the pieces are close in the Hausdorff topology) and the restrictions of $g'$ and $g$ to $V \cap V'$  are $\delta$ close in the $C^{r}$ metric .

\subsection{Semi-invariant foliations in a neighborhood a horseshoe.}\label{2}

As pointed in the introduction, complex horseshoes are important hyperbolic invariant sets appearing in automorphisms of $\C^2$, mainly because they are present whenever there is a transversal homoclinic intersection, as shown by \cite{ob}. We now give a quick review of these concepts and explain how to construct semi-invariant foliations on the neighborhood of a complex horseshoe.

Given a difeomorphism $F: M \to M$  of class $C^k$ on a riemannian manifold $M$ we say that a compact invariant set $\Lambda \subset M$ (by invariant we mean that $F(\Lambda)=\Lambda$) is hyperbolic when there are constants $C>0$, $\lambda < 1 $, and a continuous splitting $TM| \Lambda = E^s \oplus E^u$ such that: 
\begin{itemize}
    \item it is invariant: $DF_x(E^s(x)) = E^s(F(x))$ and $DF_x(E^u(x)) = E^u(F(x))$;
    \item and for any $x \in \Lambda$,  $v^s \in E^s(x)$ and $v^u \in E^u(x)$  we have \[|DF^j(v^s)|_{f^j(x)}< C\lambda^j |v^s|_x \text{ and } |DF^{-j}(v^u)|_{f^{-j}(x)}  < C\lambda^j |v^u|_x, \, \forall j \in \N;\]
\end{itemize}
where $|\cdot|_x$ is the norm on $T_xM$ associated to the riemannian metric, which we will call $d$. The bundle $E^s$ above is called the stable subbundle and the bundle $E^u$ is called the unstable subbundle.

Hyperbolic sets are useful because we have a good control on sets of points that asymptotically converge to them. For any $\varepsilon>0$ and any point $x$ in a hyperbolic set $\Lambda$ we define the \emph{stable manifold} and the \emph{local stable manifold} by 

\[ W^s(x) =\{y \in M,\lim_{n \rightarrow +\infty}d(F^n(y),F^n(x))=0\} \]
\[ W^s_{\varepsilon}(x) =\{y \in M,d(F^n(y),F^n(x))< \varepsilon, \forall n \geq 0\}\]
respectively. It is a classical result that $W^s(x)$ is a $C^k$ immersed manifold and, if $\varepsilon$ is sufficiently small, independently of $x$, $W^s_\varepsilon(x)$ is a $C^k$ embedded disk tangent to $E^s(x)$. The same results remain true for the unstable versions of the objects above, defined by considering \emph{backwards} iterates $F^{-n}$ instead of the \emph{forwards} ones above. We denote the \emph{unstable manifold} and the \emph{local unstable manifold} by $W^u(x)$ and $W^u_\varepsilon(x)$ respectively. Also, it is important to observe that they vary continuously with $x \in \Lambda$ in the $C^k$ topology and are invariant in the sense that $F(W^{s,u}(x))=W^{s,u}(F(x))$.

In the special case that $F$ is an automorphism of $\C^2$, that is, a holomorphic diffeomorphism of $\C^2$ with holomorphic inverse, the manifolds above are complex manifolds. Also, the subbundle $E^s$ and $E^u$ are such that $E^{s,u}(x)$ is a complex linear subspace of $T_x\C^2 \equiv \C^2$.

Going back to the general setting, we say that the hyperbolic set $\Lambda$ on $M$ has a \emph{local product structure} if there exists $\varepsilon > 0$ such that, for any $x,y \in \Lambda$ sufficiently close, $W^s_{\varepsilon}(x) \cap W^u_{\varepsilon}(y)$ consists of a single point $z$ that also belongs to $\Lambda$. This structure makes the neighborhood (in $\Lambda$) of a point $x \in \Lambda$ homeomorphic to the product $ (W^s_{\varepsilon}(x) \cap \Lambda) \times (W^u_{\varepsilon}(x) \cap \Lambda) $. Also, this condition is equivalent to $\Lambda$ being locally maximal, i.e., there is an open set $U$ ($\Lambda \subset U$)  such that $\Lambda=\bigcap_{n \in \Z} F^n(U)$. We say that the hyperbolic set is \emph{transitive} when there is an $x \in \Lambda$ such that $\{F^n(x), n \in \Z\}$ is dense in $\Lambda$. A hyperbolic set with these two properties is called a \emph{basic set}.

A \emph{horseshoe} is a particular type of basic set. It has the additional properties:
\begin{itemize}
    \item it is infinite,
    \item it is of \emph{saddle-type}, that is, the bundles $E^s$ and $E^u$ are non-trivial,
    \item and it is totally disconnected.
\end{itemize}

The dynamics of $F$ over a horseshoe $\Lambda$ is conjugated to a Markov shift of finite type, similarly to the \emph{Smale horseshoe}, which is conjugated to $\{0,1\}^\Z$. The last hypothesis implies that in particular a horseshoe is a zero-dimensional set and so it is topologically a Cantor set.

Horseshoes appearing in automorphisms of $\C^2$ will be called \emph{complex horseshoes} in this paper. As pointed out in the introduction, this nomenclature does not conflict with the one used by Oberste-Vorth in \cite{ob}, that is, the horseshoes constructed there are horseshoes in our sense. Another important example, to which we will refer many times in this paper, is the one constructed by Buzzard in \cite{buzz} to study \emph{Newhouse regions} in $Aut(\C^2$). The objective of this section is to show that a complex horseshoe is, locally, close to the product of two conformal Cantor sets, as defined in the previous subsection. To do so we will first need to construct stable and unstable foliations in some neighborhood of it, which is the other objective of this subsection. 

This is done in the next theorem. It is just a small adaptation of a theorem of Pixton \cite{pix} used by Buzzard in \cite{buzz}. The only difference is that we require the foliations to be $C^{1+\varepsilon}$ instead of just $C^1$. But before stating it some remarks. For a foliation $\mathcal{F}$ we will denote the leaf though a point $p$ in its domain (which is an open set) by $\mathcal{L}(p)$ and we will denote the stable and unstable foliations by $\mathcal{F}^s$ and $\mathcal{F}^u$ respectively. In the statement we will deal only with the unstable foliation but the analogous result for the stable version can be done exchanging $G$ by $G^{-1}$ and $E^s$ by $E^u$. The norm $||DG|_{E^s,E^u}||$ is equal to $\sup_{x \in \Lambda} |DG|_{E^s(x),E^u(x)}| $ and this last norm is derived directly from the euclidian metric on $\C^2$. Finally, the whole concept of horseshoe works for local diffeomorphisms, or injective holomorphisms as below.

\begin{theorem} \label{fol}Let $U \subset \mathbbm{C^2}$. Let $\Lambda \subseteq U $ be a horseshoe for an injective holomorphism $G_0: U \to M$, with $\Lambda=\displaystyle{\bigcap_{n\in \mathbbm{Z}}G_0^n(U)}$ and let $E^s \oplus E^u $ be the associated splitting of $T_{\Lambda}\mathbbm{C^2}$.

Suppose that $||DG_0|_{E^s}||\cdot || DG_0|_{E^u}||^{-1}  \cdot|| DG_0 |_{E^s}||^{-(1+\varepsilon)} <1$.

Then, there is a compact set $L$ and $\delta$ such that for any holomorphism $G : U \to \mathbbm{C^2}$ with $||G-G_0||<\delta$ we can construct a $C^{1+\varepsilon}$ foliation $\mathcal{F}_G ^u $ defined on a open set $V\subset U$ such that:

\begin{itemize}

\item the horseshoe $\Lambda_G = \displaystyle{\bigcap_{n \in \mathbbm{Z}}G^n(U)}$ satisfies $\Lambda_G \subset \text{int } L \subset L \subset \mathcal{F}_G ^u $,
\item if $p \in \Lambda_G $ then the leaf $\mathcal{L}^u_G(p)$ agrees with $W^u_{\text{loc}}(p)$,
\item if $p\in G(U)\cap U$ then $G(\mathcal{L}^u_G (p) ) \supseteq \mathcal{L}^u_G (G(p))$,i.e., it is semi-invariant,
\item the tangent space $T_p\mathcal{L}^u_\mu(p)$ varies $C^{1+\varepsilon}$ with $p$ and continuously with $G$,
\item The association $G \to \mathcal{F}^u_G$ is continuous on the $C^{1+\varepsilon}$ topology.

\end{itemize}
\end{theorem}

The proof we are going to give contains just a brief review of the argument by Pixton and then proceed to the small changes necessary to our context. It is important to observe that since $E^s(x)$ and $E^u(x)$ are complex lines in $\C^2$ and $G$ is holomorphic, then $||DG^{-1}|_{E^s,E^u}||=||DG|_{E^s,E^u}||^{-1}$ respectively. We also need an adaptation of the $C^{r} $ section theorem, which can be found in Shub \cite{global}, to the case in which the base is not overflowing. Since we did not find a precise version of what we mean by this in the literature we state the following version bellow and give a short argument on how a proof would work.

\begin{theorem}[Adapted $C^r$ section Theorem] Let $\Pi: E \to M$ be a $C^m$ vector bundle over a manifold $M$, with an admissible metric on $E$, and $D$ be the disc bundle in $E$ of radius $C$, $C>0$ a finite constant . 

Let $h:U \subset M \to M$ be an embedding map of class $C^m$ (with a $C^m$ inverse too), $U$ a bounded open set such that $ U \not\subset h(U)  $ but $h(U) \cap U \neq \emptyset $, and $F: \left.E\right|_U \to \left.E\right|_{h(U)}$ a $C^m$ map that covers $h$. 

Let also $N \subset U$ an open neighborhood of $ U\setminus h(U)$ and $s_0 : N \to \left.D \right|_N$ a $C^r$ invariant section $\left(\ceil{r} \leq m , r\in \mathbbm{R}, m \in \mathbbm{N} \right)$ . By invariant we mean that whenever $x \in N$ and $h(x) \in N$ we have $s_0\left(h(x)\right) = F\left( s_0(x) \right)$. We also need the technical hypothesis that $N \subset U\setminus h^2(U)$ and $\overline{U\setminus N} \cap \overline{U \setminus h(U)} = \emptyset$.

In this context, suppose that there is a constant $k$, $0 \leq k < 1$ such that the restriction of  $F $ to each fiber over $x \in U$, $F_x : D_x \to D_{h(x)}$ is Lipschitz of constant at most $k$, that $h^{-1}$ is Lipschitz with constant $\mu$, that $F^{(j)}$, $s_0^{(j)}$ and $ h^{(j)}$ are bounded for $0 \leq j < \ceil{k}$, $j \in \mathbbm{Z}$,  and $k \mu^{r} <1$ . Then there is a unique invariant section $s: U \to \left.D \right|_U$ (meaning that for $x \in U$ and $h(x) \in U$ we have $s\left(h(x)\right) = F\left( s(x) \right)$) with $\left.s\right|_N = s_0$ and such a section is $ C^{r}$.

\end{theorem}

\begin{proof} The loss of the overflowing condition on $h$ and $U$ is overcome by the presence of the invariant section $s_0$. The natural graph transform would carry sections over $U$ to sections over $h(U)$ but since $s_0$ is invariant in $N \supset U\setminus h(U)$ given any section $s$ that agrees with $s_0$ in $N$ we are able to extend its graph transform from $h(U) \cap U$ back to whole open set $U$. This idea comes from Robinson \cite{robin}. Besides this, very little has to be changed or verified from the proof in Shub. The admissible hypothesis on the metric works the same way to allow us to work in the context of $E=M \times A$ and write a section as $s(x)=(x,\sigma(x))$.

Next we consider the complete metric space $\Gamma \left( U, \left.D \right|_U ; s_0 \right)$ of local sections over $U$ bounded by $C$ that agree with $ s_0$ on $N' \subset U$, an open set such that $N\supset \ov{N'} \cap U \supset N' \supset U\setminus h(U)$. Careful choice of $ N' $ allow us to use a $C^{\infty}$ function $\lambda$ on $U$ that is equal to one on $N'$ and zero outside of $N$, and thusly, taking $s=\lambda \cdot s_0$ yields a well-defined section that belongs to $\Gamma \left( U, \left.D \right|_U ; s_0 \right)$; showing that it is not empty. Then consider  $\Gamma_F: \Gamma \left( U, \left.D \right|_U ; s_0 \right) \to \Gamma \left( U, \left.D \right|_U ; s_0 \right)  $ defined by: 

\[\Gamma_F(s)(x)= \begin{cases}
                s(x)\text{, if } x \in N' .\\
                F\circ s \circ h^{-1}(x) \text{, if } x \in h(U)  . \end{cases} \]
                
Since $s$ is equal to $s_0$ over $N'$, it is invariant in this open set and the definition above is coherent. Also, because $k<1$ this transformation is a contraction, so there is a unique $s$ in $\Gamma \left( U, \left.D \right|_U ; s_0 \right)$ fixed by $\Gamma_F$. It is easy to verify that this is an invariant section over $U$ that agrees with $s_0$ on $N$. 

The verification of regularity of $s$ has some minor technical differences. First, we need to verify that if $0 \leq r < 1 $ then $s$ is $r$-Hölder in all $U$.

Since $s$ agrees with $s_0$ on $N$ it is $r$-Hölder on this set, that is, for $x, y \in N$ we have $d(\sigma(x),\sigma(y)) \leq H d(x,y)^r$. Now, $U\setminus h(U)$ and $U \setminus N$ have a positive distance $\varepsilon$ between each other and the section $s$ is bounded by $C$. So, if $x \in U\setminus h(U) $ and $y\in U\setminus N$ we have $d(\sigma(x),\sigma(y)) \leq 2C \leq H \varepsilon ^r \leq H d(x,y)^r$ for some big enough constant $H$. This allow us to write $d(\sigma(x),\sigma(y)) \leq H d(x,y)^r$ for any pair $x \in U\setminus h(U)$ and $y \in U$.

As in the book we have the estimative:
$$ \displaystyle{
d(\sigma(x),\sigma(y)) \leq k^m d(\sigma (h^{-m}(x)), \sigma (h^{-m}(y))) + \ti{H} \sum_{j=1}^{m} (\mu^{r})^j k^{j-1}(d(x,y))^{r}
} $$
whenever $h^{-j}(x) \ , h^{-j}(y) \in U, \  \forall j=0,\ 1, \ 2, \dots , \ m $ .

We are going to consider two cases:

If $x,y \in U$ are such that  $h^{-j}(x) \ , h^{-j}(y) \in U, \  \forall j \in \mathbbm{N} $, we let  $m \rightarrow \infty$  in the inequality above, and, since $k \mu^{r} <1$ and $\sigma $ is bounded by $C$, the right hand side converges to $\ti{H} \cdot \ \ti{C} d(x,y)^r$. 

If else, there is a finite maximal $m$, such that $h^{-j}(x) \, h^{-j}(y) \in U, \,  \forall j=0,1,2, \dots, m  $. In this case, we can assume without loss of generality that $h^{-m}(x) \in U \setminus h(U)$. But then, using again the estimative above we have:
\begin{align*}d(\sigma(x),\sigma(y)) & \leq k^m d(\sigma (h^{-m}(x)), \sigma (h^{-m}(y))) + \ti{H} \sum_{j=1}^{m} (\mu^{r})^j k^{j-1}(d(x,y))^{r} \leq
\\ & k^m\cdot H\cdot d((h^{-m}(x),(h^{-m}(y))^r+\ti{H} \cdot \ti{C} d(x,y)^r \leq H \cdot k^m \cdot \mu ^{mr} d(x,y)^r + \ti{H} \cdot \ \ti{C} d(x,y)^r
\end{align*}
and again since $k\mu^r < 1$ we finally have $d(\sigma(x),\sigma(y)) \leq (H+\ti{H} \cdot \ \ti{C}) d(x,y)^r $ for any $x,y \in U$.

The smoothness is proved with the same argument as in the book adapted in some way as above. Using the same induction idea one can do as follows.

Consider $\ti{D}$ the disc bundle of radius $\ti{C}$ in the fiber bundle over $M$ with each fiber being equal to $L(T_xM, A)$,  $\ti{C}$ being chosen so that $||\partial{s}|| < \ti{C}$. This is admissible. 
Then the complete metric space $\Gamma(U,\left.\ti{D} \right|_U ; \partial s_0)$. of local sections that agree with $\partial s_0$ on $N'$, $\partial s$ is obtained by differentiating
$s$ on $N$ and identifying the tangent plane to $(x,\sigma_0(x))$ with the graph of a linear transformation in $L(T_xM, A)$.
The graph transform $\gamma_{DF}(\tau)$ is defined by: 

\[\gamma_{DF}(\tau)(x)= \begin{cases}
				\partial s(x), \text{ if } x \in N',\\
				\Gamma_{DF} \circ s \circ h^{-1}(x), \text{ if } x \in h(U)   ,
			\end{cases}\]

where $\Gamma_{DF(L)}:= (\Pi_2 DF_{x,\sigma(x)})\circ (Id,L) \circ Dh^{-1}_{h(x)} $, for any $L$ a linear transformation in $L(T_xM, A)$,  is a fiber contraction of constant $k\mu <1$. 

To show that the invariant section $\ti{\tau}$ is indeed the tangent to $(x,\sigma(x)), x \in U$ we have to divide in cases as above:

If $x \in  \bigcup_{n \in \mathbbm{N}} h^n(N')$ then it is true by definition of $ \partial s$ and the fact that $\ti{\tau} $ is invariant and equal to $\partial s_0$ on $N'$ (remember that $s_0$ is $C^r$)

If not, then for any $n \in \mathbbm{N}$ there is $\delta$ small enough such that if $d(x,y) < \delta$ then $h^{-j}(x), h^{-j}(y) \in U$ for $j=0,1,2, \dots, n$. This comes from the fact that $x \in \bigcap_{i \in \mathbbm{N}}h^i (U)$ and $h^n(U)$ is an open set around set $x$ . This is enough to show, by the same iteration argument, that $\text{Lip}_0(\sigma(x+y), \sigma(x) + \ti{\tau}(x)(y))=0$, which completes the proof.

\end{proof}

\begin{remark} Observe that, from the argument above,  if we just want to obtain an invariant section that is continuous we can just make $m=r=0$ and consider just the case in which $M$ is a topological space rather than a manifold.

\end{remark}

\begin{remark}\label{depen}

If we add the additional hypothesis that the maps $h \text{, } F \text{, } s_0$ and all their derivatives are uniformly continuous, the proof above also shows that the invariant section varies continuously with the maps involved. More specifically, fixing $h, F, s_0$ and choosing  any $h',F',s'_0$ such that $h \text{ and  } h'$, and their inverses are $C^m$ close;  $F \text{ and } F'$ are $C^m$ close (and $F'$ covers $h'$); $s_0 \text{ and } s'_0$ are invariant (by $F \text{ and } F' $ respectively) and $C^r$ close; and $k \mu ^r < 1$ then $s$ and $s'$ are both close in the $C^ {r}$ topology. The proof is essentially the same as above and the details are left to the reader.

\end{remark}

We now proceed to the proof of Theorem \ref{fol} 

\begin{proof} (\textbf{Theorem \ref{fol}}) The work of Pixton shows that we can construct a non necessarily smooth $\mathcal{F}^u_G$ for any $G$ with the desired properties as above. The idea is described as follows.

We begin by constructing a transversal (not necessarily semi-invariant) foliation $\mathcal{F}_0$ to $W^s_G(\alpha)$ that covers an open set around $W^s_G(\alpha)$. Here $W^s_G(\alpha)$ denotes the union of all $W^s_{G,\alpha}(p)=W^s_G(p)\cap B_\alpha(p)$ for some sufficiently small $\alpha$. This can be done locally and, in the case that $W^s_G(\alpha)$ is a zero dimensional transversal lamination, which is our case, it is possible to glue these constructions together by bump functions (check the original for details). We can restrict $\mathcal{F}_0$ to small neighbourhood $V$ of  $\ov{W^s_G(\alpha)\setminus G(W^s_G(\alpha))}$, in such a way that $G(V)\cap V = V'$ does not intersect $G^{-1}(V)\cap V = G^{-1}(V')$. We consider a new foliation $\mathcal{F}'_0$ on $V' \cup G^{-1}(V')$ defined being the same as $\mathcal{F}_0$ over $G^{-1}(V')$ and being equal to $G(\mathcal{F}_0)$ over $V'$. We can then, considering again that $W^s_G(\alpha)$ is transversely zero dimensional, construct a transversal foliation $\mathcal{F}_1$ to it that agrees with $\mathcal{F}'_0$ on $V \cup G^{-1}(V')$. Now we define recursively $\mathcal{F}_n = (G(\mathcal{F}_{n-1}) \cap U) \cup ((\mathcal{F}_{n-1} \cap V)  $, which is possible because of the semi-invariance of $\mathcal{F}_1$. Notice that for any point $x \in U \setminus \bigcap _{n \in \mathbbm{N}} G^{n}(U) $ for any integer $n$ bigger than a integer $n_x$ the leaf $\mathcal{L}_n(x)$ of $\mathcal{F}_n$ at $x$ is the same so we can safely define in $U \setminus \bigcap _{n \in \mathbbm{N}} G^{n}(U)$ the limit foliation $\mathcal{F}$. Finally, adding the submanifolds $W_u (x), x \in \Lambda $, yields a invariant foliation in an open subset of $U$ (also see \cite{mpps} for the idea of fundamental neighbourhood). Notice that we can chose $L \text{ and } \delta $ small enough such that the items above are satisfied for any $G\text{, }||G-G_0|| < \delta$.

We can use the $C^r$ section Theorem to show that this foliation is indeed $C^{1+\varepsilon}$. Begin by extending the fibrate decomposition $E_{G_0}=E=E^s \oplus E^u$ over $\Lambda$ to a close $C^2$ decomposition  $E=E^s \oplus E^u$ over $U$ such that the action of the derivative map $TG_x: =E_x^s \oplus E_x^u \to E_{G(x)}^s \oplus E_{G(x)}^u$ can be written as a block matrix:

$$\begin{bmatrix}
    A_x & B_x \\
    C_x & D_x 
\end{bmatrix}$$

in which $||A_x|| <||DG_0|_{E^s}||+\delta '$, $||D_x|| > ||DG_0|_{E^u}||-\delta '$, and $||B_x||, ||C_x|| < \delta '$ for some small $\delta' $ uniformly on  $U$. Also, by possibly shrinking $U$ we may assume that the tangent directions to $\mathcal{F}$ can be written as the graph of a linear map from  $E_x^u$ to $E_x^s$  (bounded uniformly on $U$). Considering the $C^2$ bundle whose fibers are $L(E_x^u,E_x^s)$ we can consider the covering map :

$$\Gamma_{DF}(x)(L) )= [B_x + A_x L)][D_x + C_x L]^{-1} $$

in which $L\in L(E_x^u,E_x^s) $. By making $\delta \text{ and } \delta'$ sufficiently small we have that $\Gamma_{DF}$ is a fiber contraction of constant at most $ ||DG_0|_{E^s}|| \cdot ||DG_0^{-1}|_{E^u}|| + \delta''$.  The Lipschitz constant of the base map $G^{-1}$ is at most $ ||DG_0|_{E^s}||^{-1}  + \delta'''$, and so there is an $r>0$ such that \[(||DG_0|_{E^s}|| \cdot ||DG_0^{-1}|_{E^u}|| + \delta'') \cdot (||DG_0|_{E^s}||^{-1}  + \delta''')^r <1.\] This is enough to show that the section $x,T_x(\mathcal{F}\cup W^u)$ is the unique invariant section of the $C^r$ section Theorem that agrees with $\mathcal{F}$ on $N$, and so it is $C^{1+\varepsilon}$. By the same argument on the Fröbenius theorem we can express the foliation $\mathcal{F}$ locally through a finite number of $C^1$ charts and the fact that the section above is $C^{1+\varepsilon}$ allows us to show that these charts are actually $C^{1+\varepsilon}$. The continuity in the $C^{1 + \varepsilon}$ topology comes immediately from the construction and previous observations, we only require $\mathcal{F}_0$ and its derivatives to be uniformly continuous on $V$ which is clearly possible to be done.

\end{proof}

\begin{remark} Whenever we have a hyperbolic set $\Lambda$ there is an adapted metric $||\cdot||'$ such that the constant $C$ in the definition of hyperbolic set is equal to $1$. In this metric the condition $||DG_0|_{E^s}||\cdot || DG_0|_{E^u}||^{-1}  \cdot|| DG_0 |_{E^s}||^{-(1+\varepsilon)} <1$ will be automatically satisfied, for some $\varepsilon$ sufficiently small, with $||\cdot||'$ instead of $||\cdot||$. Since such metrics are uniformly equivalent in a compact set containing both the foliations above it follows that close to a complex horseshoe we can always construct stable and unstable foliations with the properties listed on \textbf{Theorem \ref{fol}}.

\end{remark}  

\begin{corollary}
With the hypothesis $||DG_0|_{E^s}||\cdot|| DG_0 |_{E^u}|| <1$ the last theorem guarantees the existence of a $C^2$ foliation $\mathcal{F}^u_G$ for any $G$ sufficiently close to $G_0$.
\end{corollary}

This could be the case in the dissipative context, specially in the case of horseshoes arising from transversal homoclinic intersections.

\begin{remark} \label{hololeave} Each leaf of the foliation obtained in theorem \ref{fol} can be chosen to be a holomorphic curve. This only depends on being able to consider the foliation $\mathcal{F}^1$ consisting of leaves that are holomorphic curves. The local construction of $\mathcal{F}^1$ in \cite{pix} involves only an isotopy and a bump function applied to create disk families along compact (and possibly very small) parts of $W^s$. Checking the details in the original, we observe that such construction can be done in a way that make those disk families be holomorphic embedded curves. This is mentioned in \cite{buzz}; see the appendix of \cite{buzzthesis} for further details.

\end{remark}

\subsection{Conformal Cantor sets locally describe horseshoes.}\label{3}

To end this section we show that a horseshoe is, locally, close to the product of two conformal Cantor sets. Having in mind the local product structure this fact is a consequence of the following theorem.

\begin{theorem*a} \label{h=c} Let $\Lambda$ be a complex horseshoe for a automorphism $G \in Aut(\C^2)$ and $p$ be a periodic point $\Lambda$. Then, if $\varepsilon$ is sufficiently small, there is a parametrization  $\pi: U \subset \C \to W^u_\varepsilon(p)$ such that $\pi^{-1}(W^u_\varepsilon(p) \cap \Lambda)$ is a conformal Cantor set in the complex plane.

\end{theorem*a} 

Of course an analogous version is true for the stable manifold. The main ingredient is the following lemma.

\begin{lemma}\label{projections}
Let $\Lambda_G$ be a complex horseshoe for a automorphism $G \in Aut(\C^2)$ together with its unstable foliation $\mathcal{F}^u_G$. Additionally, let $N_1$ and $N_2$ be two $C^{1+\varepsilon}$ transversal sections to $\mathcal{F}^u_G$. Suppose that for some point periodic point $p \in \Lambda_G$, the tangent planes of $N_1$ and $N_2$ to the points of intersection $N_1 \cap W^u_G(p) =q_1$ and respectively $N_2 \cap W^u_G(p)=q_2$ are complex lines of $\C^2$. The the projection along unstable leaves $\Pi_u : N_1 \to N_2$ is a $C^{1+\varepsilon}$ map conformal at $q_1$.
\end{lemma} 

\begin{proof}
Observe that, since $p \in \Lambda_G$ every backwards iterate of the segment in $W^u_G(p)$ that connects $q_1$ and $q_2$ stays on the domain of the foliation. So, for every $n \in \N$ we can define a restriction $N_i^n \subset N_i, \, i=1,\,2.$ such that $G^{-n}(N^n_i)$ is also on the domain of the foliation. Furthermore, this restriction can be done in such manner that, since $p$ is periodic, we have, by the $\lambda$-lemma, that $G^{-n}(N^n_1)$ and $G^{-n}(N^n_2)$ are $\delta$ close to each other on the $C^1$ metric, for every $n > n_\delta$. Also, we can assume that their tangent directions at $q^n_i=G^{-n}(q_i)$ are bounded away from $T_{q^n_i}W^u_G$, $i=1,\,2.$ Let $\Pi^n_u : N^n_1 \to N^n_2$ be the projection along the unstable foliation.

Looking at a small open set $\ti{U}$, we can find a $C^{1+\varepsilon}$ map $f:\ti{U} \to \mathbbm{D}\times\mathbbm{D}$ such that the unstable leaves are taken into the horizontal levels $\mathbbm{D} \times \{z\} \text{, } z \in \mathbbm{D}$ and represent $N^n_1$ and $N^n_2$ as graphs $(h_1(z), z)$ and $(h_2(z),z)$ of $C^{1+\varepsilon}$ embeddings $h_1 \text{ and } h_2$ with domain being a small disk $\mathbbm{D_\varepsilon}$ too. Under this identification, $\Pi^n_u$ is a $C^{1+\varepsilon}$ map that carries  $(h_1(z), z)$ to $(h_2(z),z)$, and, according to the previous paragraph conclusion, has a  derivative $\delta$ close to the identity.

Now, the projection along unstable foliations comute with $G$. Therefore, $\Pi_u = G^{n} \circ \Pi^n_u \circ G^{-n}$. Using the chain rule to calculate the derivative at $q_1$ we obtain a  expression of the form

$$A_1 \cdot A_2 \cdots A_n \cdot D\Pi_s^n \cdot B_n \cdots B_1$$

where $B_i$ represents the restriction of $(DG)^{-1}$ to $T_{q_1^n}N_1^n$ and $A_i$ the restriction of $(DG)^{1}$ to  $T_{q_2^n}N_2^n$. But all of these tangent spaces are, by induction, complex lines in $\C^2$, so all the $A_i$ and $B_i$ are conformal.This way,  the derivative of $\Pi_u$ is at most $\delta$ distant from being conformal. Making $\delta \rightarrow 0$ (or equivalently, $n \rightarrow \infty$) we have the desired conformality.

\end{proof}

The proof of \textbf{Theorem A} will be done using the Buzard's horseshoe \cite{buzz} since it makes the comprehension easier and we will need this example later. For the general case one need just to use \emph{Markov neighborhoods} as in \cite{pix}, but the proof is easily deduced from the proof for this example. So now we proceed to a brief recapitulation of this example and construct larger Markov neighborhoods for it.

\begin{example}\label{shoe}(Buzzard)
Let $S(p;l) \subset \C$ denote the open square centered at $p$ of sides parallel to the real and imaginary axis of side length equal to $l$. Consider the 9 points set $P= \{x+yi \in \C; \, (x,y) \in \{-1,\,0,\,1\}^2 $ and a positive real number $\delta<1$. Define $c_0 = 1-\delta$ and:

$$K_0 = \displaystyle{\bigcup_{a \in P} \ov{S(a;c_0)}} \text{ and } K_1 = K_0 \times K_0 \subset \C^2.$$

We identify each connected component of $K_1$, $ S(a;c_0) \times S(b;c_0) $  as the pair $(a,b) \in P^2$.

Consider now, some positive real number $c_1 \in (c_0 =1-\delta, \frac{3c_0}{2+c_0 }=\frac{3-3\delta}{3-\delta})$ and the map $f: K_0 \to \mathbbm{C}$ defined as,

$$f(w):= \sum_{a \in P}\frac{3a}{c_1}\chi_{\ov{S(a;c_0)}}(w).$$

Notice that its image is composed of nine points as is $P$. Analogously, we can define $K_g = \displaystyle{\bigcup_{a \in P} \ov{S(\sfrac{3a}{c_1};3)}} $ and define,

$$g(z):= \sum_{a \in P}-a\cdot\chi_{\ov{S(\frac{3a}{c_1};3)}}(z).$$

Then defining the maps,

\begin{align*}
    F_1(z,w)& : = (z+f(w), w)\\
    F_2(z,w) &:=(z,w+g(z))\\
    F_3(z,w) &:= \left(\frac{c_1}{3}z,\frac{3}{c_1}w\right)
\end{align*}

and making $F:K_1 \to \C^2; \, F:=F_3 \circ F_2 \circ F_1$ we have that in a connected component $(a,b)$ of $K_1$

$$F(z,w) = \left(\frac{c_1}{3}z+b, \frac{3}{c_1}(w-b)\right).$$

The maximal invariant set of $F$ over $K_1$, $\Lambda = \bigcap_{n \in \mathbbm{Z}} F^n(K_1) $, is a hyperbolic set with $0$ as a fixed saddle point. It is easy to see that $W^u_{F,\text{loc}}((0,0)) := \{0\} \times \{ S(0; c_0) \} $ is the connected component that contains $(0,0)$ of the intersection between $W^u_F(0)$ and the connected component $(0,0)$ of $K_1$. Also, the set $W^u_{F,\text{loc}}(0) \cap \Lambda$ can be seen as a conformal Cantor set $K_F$ on the complex plane (in this case $0 \times \C$) given by the maps:

\begin{align*}
    g_a: S(a;c_0) & \to S(0; 3)\\
            z & \mapsto \frac{3}{c_1}(z-a)
\end{align*}

Likewise, we can write  $W^s_{F,\text{loc}}(0) \cap \Lambda$ as the same cantor set $K$. The condition $c_1 < \frac{3c_0}{2+c_0}$ is necessary for the image of each $g_a$ cover the union of their domains.

Now we work with automorphisms of $\mathbbm{C}^2$ that are sufficiently close to this model $F$. First, we approximate $f$ and $g$ by polynomials $p_f$ and $p_g$, obtaining a map $G_0 = F_3  \circ F'_2 \circ F'_1 \in Aut(\C^2)$, where $F'_1(z,w) : = (z+p_f(w), w)$ and $F'_2(z,w) :=(z,w+p_g(z))$. Then, we fix $K' \subset \ov{K'} \subset \text{int}(K_1) $ such that considering $\Lambda_G$ the maximal invariant set by $G$ of the open set $U$ it is contained in $K'$ whenever $||G-G_0|||_{U}$ is sufficiently small, where $U=S(0;3)\times S(0;3) $. Furthermore, there is a fixed point $p_G$ that is the analytic continuation of the fixed point $(0,0)$ of $F$. Since $||G-F||$ is small we can also show that the projection $\Pi: W^s(p_G; loc) \to S(0;3)$ is a biholomorphic map close to the identity, where $W^s(p_G; loc)$ is the connected component that contains $p_G$ of $W^s(p_G) \cap U$ (notice it is a larger portion of the unstable manifold then previously defined).

Observe that  $G^{-1}(W^s(p_G; loc)) \cap S(0;\frac{3}{c_1}c_0)\times K_0 $ is made of nine different connected components, $W_1, W_2, W_3, \dots , W_9$ , each of them holomorphic curves close to being horizontal, because of the continuity dependence of the foliations on $G$ (so, as long as $f$ and $g$ are well approximated by $p_f$ and $p_g$ and $||G-G_0||$ is sufficiently small). Consider now $V_i=G (W_i \cap K'), i=1,2,\dots,9 $. Notice that all the $V_i $ are disjoint subsets of $W^s(p_G; loc)$.

According to theorem \ref{fol} and remark \ref{hololeave},  $\mathcal{F}^u_G$ can be defined whenever $G$ is sufficiently close to $F$ and we can consider its leaves to be holomorphic lines very close to the vertical lines. However, its domain may be only a small neighbourhood of $\Lambda_G$. We now show a way of constructing it that cover a large subset of $U$. 

First, consider the foliation by vertical leaves $\{z\} \times S(0;3)$ defined for $z$ on a small neighborhood of $\overline{S(0;3)}$. If $G$ is sufficiently close to $F$, there is a real number $k>1 $ (close to $1$) such that:
$$ S(a; k^{-1}c_0  ) \times S(0;3) \subset G(S(0;3) \times S(a; kc_0  )) \subset S(a; kc_0  ) \times S(0;3) \forall a \in P $$
$$\text{ and }  G(U) \cap U \subset \bigcup_{a \in P} G(S(0;3) \times S(a; kc_0  )) $$

Let $V_{-1}(a)= S(a; k^{-1}c_0  )  $ and $V_1(a)=S(a; kc_0  )$ as above. Also, let \[V(a) = G(S(0;3) \times S(a; kc_0  )) \cap \{w=0\}\]
Then $V_{-1}(a)\subset V(a) \subset V_{1}(a)$. The image of the vertical foliation restricted to $S(0;3) \times S(a; kc_0  )$  by $G$ is a foliation of $G(S(0;3) \times S(a; kc_0  ))$ described as  $(u,v)\mapsto(u+\Psi_a(u,v), v)$ for $u \in V(a)$ and $v \in S(0;3)$ (after an obvious shrinking), with $\Psi_a$ small in the $C^1$ metric as $||G-F||$ is small. Notice that $\Psi_a(u,v)$ is always holomorphic on $v$, which is equivalent to the fact that the leaves, given by $(u_0+\Psi_a(u_0,v),v)$ for $u_0$ fixed, are holomorphic curves.

For each $a \in P$ fix $\lambda: \C \to [0,1]$ a bump function with support contained in $V_1(a)$ and such that $ V_{-1}(a)\subset \{\lambda_a(z) = 1\} \subset V(a)$. It is easy to see that $\lambda_a$ with these properties can be chosen independently of $G$. We can now extend each of the foliations above to $V_{1}(a) \times S(0;3) $ by:
\[(u,v) \mapsto (u+\lambda_a(u)\cdot\Psi_a(u,v), v),\]
which yields a foliation that is $C^{\infty}$ with holomorphic leaves (for each fixed $u_0$). By choosing $\Psi_a(u,v)$ sufficiently small (relatively to $||\lambda_a||_{C^1}$) we can guarantee that the map above is injective. To guarantee that it is surjective we need only to observe that it is clearly surjective outside of $\supp \lambda_a$ and for $u' \in \supp \lambda$ the map
\[u \mapsto u' - \lambda_a(u)\cdot\Psi_a(u,v)\]
maps a set homeomorphic to the closed ball (and contained in $V_1(a)$ inside itself, hence has a fixed point.

Finally we can consider a foliation given by
\[(u,v)\mapsto (u+\sum_{a \in P}\lambda_a(u)\cdot\Psi_a(u,v), v) \]
for $(u,v) \in S(0;3) \times S(0;3)$. Restricting it to an open subset we get a foliation with the same properties of $\mathcal{F}_1$ in the proof of theorem \ref{fol} and repeat the construction to obtain the foliation $\mathcal{F}_G^u$.

In view of the continuous dependence of the foliation on $G$, and maybe by restricting the foliation to an open set, we can assume that the leaves of $\mathcal{F}^u_G$ are almost vertical. Thusly, we can define the projections along stable leaves $\Pi_i : W_i \to   W^s(p_G; loc) $.

\end{example}

\begin{proof} (\textbf{Theorem A})
We need to show that can express $K_G=\Pi(W^s(p_G; loc) \cap U \cap \Lambda_G)$ as a dynamically defined conformal Cantor set through the maps $f_i : \Pi(V_i) \to S((0,0);  3) $ where $f_i= \Pi \circ \Pi_i\circ G^{-1} \circ \Pi^{-1}$. Let us show that $K_G$ is the maximal invariant set of these maps. Take $x \in W^s(p_G; c_0) \cap U \cap \Lambda_G$. Thus, $G^{-1}(x) \in \Lambda \subset U$ , so there exists $i \in \{1,2,3, \dots, 9\}$ such that $G^{-1}(x) \in W_i $, which implies $x \in V_i$. Likewise,  $y=\Pi_i(G^{-1}(x)) \in \Lambda$. To show this,  we see that $G^{n}(y) \in W^s(p_G; c_0) \cap U, \text{ for all } n\geq 0 $, as this set is carried into itself by forward iteration of $G$. Additionally, $G^{-n}(y) \in U \text{ for all } n > 0 $  because $y \in W^s(G(x))$ and backwards iterations of unstable manifolds always remain inside $U$ by construction. So, $y \in \bigcap_{n \in \mathbbm{Z} G^n(U)}=\Lambda_G$, and in particular $y \in W^s(p_G; c_0) \cap U$. Hence, as we have already shown, $y \in V_i \text{ for some } i \in \{1,2,3,\dots,9\}$. Repeating this argument inductively we obtain that the orbit of $\Pi(x)$ always remains on $\bigcup_{i=1}^9 V_i$. 

On the other hand, if $x \in W^s(p_G; c_0) \cap U$ is such that the forward orbit of $\Pi(x)$ by the maps $f_i$ is always in $ \bigcup_{i=1}^9 V_i$, then, using that projections along the unstable leaves commute with the map $G$ and denoting by $x_n$ the $n-\text{th}$ term of the orbit of $x$ by the $f_i$, we can show that $G^{-n}(x)=\Pi_u\circ\Pi^{-1}(x_n) \text{, } (n > 0)$, $\Pi_u$ being a projection along unstable leaves between two close components of $W^u(p_G) \cap U$. This implies that $G^{-n}(x) \in U \text{, } \forall n >0$, and as $G^{n}(x) \in U, \forall n \geq 0$, then $x \in \Lambda_G$.

It is clear that the manifolds $W^s_G(p_g,loc)$ and $W_i$ satisfy the properties of the transversal sections on the lemma just above. It is then clear that the maps $f_i$ are $C^{1+\varepsilon}$ and conformal at $K$ (notice that $\Pi$ is a parametrization of a complex line). 

The general case of a complex horseshoe can be done using Markov neighborhoods, as described in \cite{pix}. The improvement from the work of Bowen \cite{bowen} is that the rectangles are open sets of the ambient space filled with our stable and unstable foliations. Letting $R_j, \, j=1, \dots,m$ be the Markov partition of $\Lambda_G$ we consider $W(G)$, a large compact part of $W^s_G(p_G)$ of some fixed point $G$ that has only one connected component intersection with $R_j$ for all $j=1, \dots,n$. Then, define the sets $G(i,j)$ as $G^{-1}(R_i \cap W(G) )$ and the maps 
\begin{align*}
     g_{(i,j)} & : G(i,j) \to W(G) \\
    q & \mapsto \Pi^s_j (G(q))
\end{align*}

for all $i,\,j=1,\dots,n$, where $\Pi^u_j$ denotes the projection along the stable leaves \textit{inside} $R_j$. Notice that in this case the previous need to extend the foliations disappears given the presence of the Markov partition.

Verifying that this set of data defines a dynamically defined Cantor set follows from the arguments on the example above almost \say{\textit{ipsis literis}}. This finishes the proof.

\end{proof}


\begin{remark}
One can also observe, that taking $p_f$ and $p_g$ sufficiently good approximations and requiring $||G-G_0|_{U}||$ to be sufficiently small, the Cantor set obtained above, identified as $K_G$ is in a small open neighbourhood $\mathcal{V}$ of $K_F$ in $\Om_{{P}^{\N}}$. This will be important in section \ref{last}.

\end{remark}
\section{A sufficient criterion for the stability of conformal Cantor sets.}
In this section we explore the consequences of the conformality on the structure of the Cantor sets. The first result is the existence of limit geometries and some consequences of it. In the second subsection we define renormalization operators and verify that the limit geometries are an atractor with respect to their actions. In the last section, we use this last fact to show that the concept of recurrent compact is a sufficient criterion for the stability of intersections between Cantor sets. All of these concepts and techniques are natural extensions from the real case.

\subsection{Limit Geometries}\label{4.1}

Given a conformal Cantor set $K$ we define $K(a)=K \cap G(a)$ and fix $c(a)\in K(a)$ for all $a \in \mathbb{A}$. Additionally, given $\ul{\theta} = ( \dots, \theta_{-n}, \dots, \theta_0 ) \in \Sigma^- $ we write $\ul{\theta}_n = (\theta_{-n}, \dots, \theta_0  )$ and $r_{{\ul{\theta}}_n}:= \text{diam}(G^*(\ul{{\theta_n}}))$.

As previously mentioned we can extend $g$ and its inverses to a neighbourhood of $\bigsqcup_{a\in\mathbb{A}}G(a)$, so we may consider, in the case that $(a_i,a_{i+1}) \in B $, $f_{(a_i,a_{i+1})}$ defined from $ G^*(a_{i+1})$ to $G^*(a_i) $; and hence also consider $f_{\ul{a}}: G^*(a_0) \to G^*(\ul{a})$ when $\ul{a}\in \Sigma^{fin}$. With this in mind we can define, for any $\ul{\theta} \in \Sigma^-$ and $n \geq 1$: 
$$ c_{\ul{\theta}_n}=f_{\ul{\theta}_n}(c_{\theta_0})$$
$$k^{\ul{\theta}}_n= \Phi_{\ul{\theta}_n} \circ f_{{\ul{\theta}_n}}$$

where $k^{\ul{\theta}}_n: G^*(\theta_0) \to \C $ and  $\Phi_{\ul{\theta}_n}$ is the affine transformation over $\C$, $\Phi_{\ul{\theta}_n}(z)=\alpha \cdot z + \beta, \alpha \in \C^*, \beta \in \C$, such that $\Phi_{\ul{\theta}_n}( c_{\ul{\theta}_n})=0$ and $ D(\Phi_{\ul{\theta}_n} \circ  f_{{\ul{\theta}_n}})(c_{\theta_0})=1\in \C $. A transformation with these properties exists because the map $g$, and thusly its inverse branches, are conformal on the set $K$, so $Df_{{\ul{\theta}_n}}(c_{\theta_0})$ is a conformal matrix that can be seen as a linear operator over $\C$, or precisely, a multiplication by a complex number. We denote the space of affine transformations over $\C$ by $Aff(\C)$.

Define $\Sigma^-_a = \{\ul{\theta} \in \Sigma^-, \ul{\theta}_0 = a \}$ and consider in this set the topology given by the metric $d(\ul{\theta}_1,\ul{\theta}_2) = \text{diam}(G^*(\ul{\theta}_1 \wedge \ul{\theta}_2)) $. Likewise, let the space $\text{Emb}_{1+\varepsilon}(G^*(a), \C)$ of $C^{1+\varepsilon}$ embeddings from $G^*(a)$ to $\C$ with $C^{1+\varepsilon}$ inverse have the topology given by the metric $d(g_1,g_2)= \max\{||g_1-g_2||,\log||D(g_1 \circ (g_2)^{-1})||\}$. Notice that this metric is equivalent to the usual $C^{1+\varepsilon}$ metric. With these notations and considerations we have the following lemma.

\begin{lemma}\label{limgeo} (Limit Geometries) For each $\theta \in \Sigma^-$ the sequence of $C^r$ embeddings $k^{\ul{\theta}}_n: G^*(\theta_0) \to \C $ converges in the $C^{r}$ topology to an embedding $ k^{\ul{\theta}}: G^*(\theta_0) \to \C $. Moreover, the convergence is uniform over all $\ul{\theta} \in \Sigma^-$ and in a small neighbourhood of $g$ in $\Om^r_{\Sigma}$. The map $k: \Sigma^-_a \to \text{Emb}(G^*(a), \C), \; \ul{\theta} \mapsto k^{\ul{\theta}}$ is Hölder, if we consider the metrics described above for both spaces. The $ k^{\ul{\theta}}: G^*(\theta_0) \to \C $ defined for any $\theta \in \Sigma^-$ are called the limit geometries of $K$.

\end{lemma}

\begin{proof}

We will first prove the result for $1 \leq r < 2$. Consider for each $n \geq 2$ (and $\ul{\theta} \in \Sigma^-$) the functions $\Psi^{\ul{\theta}}_n: \text{Im}(k^{\ul{\theta}}_{n-1}) \to \C$:
$$\Psi^{\ul{\theta}}_n= \Phi_{\ul{\theta}_n} \circ f_ {\left(\theta_{-n}, \theta{-n+1}\right)}\circ \Phi_{\ul{\theta}_{n-1}}^{-1}.$$
Notice that, then:
\begin{equation} \label{eq:psi}
k^{\ul{\theta}}_n=\Psi^{\ul{\theta}}_n \circ \Psi^{\ul{\theta}}_{n-1} \circ \dots \circ \Psi^{\ul{\theta}}_2 \circ k^{\ul{\theta}}_1
\end{equation}
We proceed by controlling the functions $\Psi^{\ul{\theta}}_n$ and showing that they are exponentially close to the identity.

First, the domain of $f_{(\theta_{-n}, \theta_{-n+1})}$ in the definition of $\Psi^{\ul{\theta}}_n$ is $G^*(\theta_n)$. Denoting its diameter by $r_{\ul{\theta}_{n}}$, we know that $r_{\ul{\theta}_{n}} \leq C \cdot \mu^{-n}$ for some constant $C>0$, as shown in lemma \ref{size}. But we can do better. In what follows all the $C$'s (with subscript, superscript, or without them) will always denote a positive real constant that will in some way depend on the other constants previously appearing on this proof, but never on $n \in \N$, nor $\ul{\theta}$.  If a map $f$ is $C^{1+\varepsilon}$ on an open subset of $\C$, then, for any point $z \in \C$ and $h \in \C$ small enough:

\[  ||f(z+h)-(f(z)+Df(z)\cdot h)|| < C||h||^{1+\varepsilon} \]

Consequently, $ f_{(\theta_{-n}, \theta_{-n+1})}:G^*(\ul{\theta}_{n-1}) \to G^*(\ul{\theta}_n)$ is $C_f \cdot r_{\ul{\theta}_{n-1}}^{1+\varepsilon}$ close to the map $A_{\ul{\theta}_n} \in Aff(\C)$ described by \[A_{\ul{\theta}_n}(c_{\ul{\theta}_{n-1}})=c_{\ul{\theta}_n} \text{ and } DA_{\ul{\theta}_n} = Df_{(\theta_{-n}, \theta_{-n+1})}(c_{\ul{\theta}_{n-1}}),\] and thus, if $n$ is large enough:
 \[r_{\ul{\theta}_{n}} \leq |Df_{(\theta_{-n}, \theta_{-n+1})}|\cdot r_{\ul{\theta}_{n-1}} + C_f \cdot r_{\ul{\theta}_{n-1}}^{1+\varepsilon} \leq  r_{\ul{\theta}_{n-1}} \cdot \left( |Df_{(\theta_{-n}, \theta_{-n+1})}| + C_1 \cdot \mu^{(-n+1)\varepsilon}\right)\] 

Arguing by induction, we obtain:
\begin{gather*}
    \log r_{\ul{\theta}_{n}} \leq \log |Df_{\ul{\theta}_n}(c_{\theta_0})| +  ||Dg|| \cdot \sum_{j=0}^{n-1} \left(  C_1 \cdot \mu^{(-n+1)\varepsilon} \right) \leq \log |Df_{\ul{\theta}_n}(c_{\theta_0})| + C_2 ,\text{ so }  \\  r_{\ul{\theta}_{n}} \leq C' \cdot |Df_{\ul{\theta}_n}(c_{\theta_0})| \leq C' \cdot \mu^{-n} 
\end{gather*}
    
In a completely analogous way we can show, maybe enlarging $C'$, that:
\begin{equation} \label{eq:size}
    {C'}^{-1} \cdot |Df_{\ul{\theta}_n}(c_{\theta_0})|   \leq  r_{\ul{\theta}_{n}} \leq C' \cdot |Df_{\ul{\theta}_n}(c_{\theta_0})| 
\end{equation}

and so the size of the $G^*(\ul{\theta}_n)$ is controlled. This implies that  $ ||f_{(\theta_{-n}, \theta_{-n+1})}-A_{\ul{\theta}_n}|| \leq C \cdot |Df_{\ul{\theta}_n}(c_{\theta_0})|^{1+\varepsilon} $ for some constant $C$, for all $\ul{\theta} \in \Sigma^-$. On the other hand, by construction $ \Phi_{\ul{\theta}_n} \circ A_{\ul{\theta}_n} \circ \Phi_{\ul{\theta}_{n-1}}^{-1} = Id$ and $ D\Phi_{\ul{\theta}_n} = \left(Df_{{\ul{\theta}_n}}(c_{\theta_0})\right)^{-1} $, therefore: 
\begin{align*}
||\Psi^{\ul{\theta}}_n -Id|| & = ||\Psi^{\ul{\theta}}_n -\Phi_{\ul{\theta}_n}\circ A_{\ul{\theta}_n} \circ \Phi_{\ul{\theta}_{n-1}} || \leq |D\Phi_{\ul{\theta}_n}| \cdot || f_{(\theta_{-n}, \theta_{-n+1})}-A_{\ul{\theta}_n} ||  \\
& \leq |Df_{{\ul{\theta}_n}}(c_{\theta_0})|^{-1} \cdot C \cdot |Df_{\ul{\theta}_{n-1}}(c_{\theta_0})|^{1+\varepsilon}\\
& \leq C_3 \cdot {(\mu^{-\varepsilon})}^n
\end{align*}
as we wished to obtain.
This is enough to show that $\{\Psi^{\ul{\theta}}_n\}_{n\geq0}$ is a Cauchy sequence, at least in $C^0$ metric. In fact, for $m,l \geq 1$, 
\begin{align*} 
||\Psi^{\ul{\theta}}_{m+l}-\Psi^{\ul{\theta}}_m|| &= || \Psi^{\ul{\theta}}_{m+l} \circ \dots \circ \Psi^{\ul{\theta}}_2 \circ k^{\ul{\theta}}_1 - \Psi^{\ul{\theta}}_{m} \circ \dots \circ \Psi^{\ul{\theta}}_2 \circ k^{\ul{\theta}}_1  ||\\
& \leq \sum_{j=1}^l || \Psi^{\ul{\theta}}_{m+j} \circ \dots \circ \Psi^{\ul{\theta}}_2 \circ k^{\ul{\theta}}_1 - \Psi^{\ul{\theta}}_{m+j-1} \circ \dots \circ \Psi^{\ul{\theta}}_2 \circ k^{\ul{\theta}}_1 || \leq  \sum_{j=1}^l || \Psi^{\ul{\theta}}_{m+j} -\text{Id} || \\
& \leq  \sum_{j=1}^l C_3 \cdot {(\mu^{-\varepsilon})}^{m+j} \leq \frac{C_3 \cdot (\mu^{-\varepsilon})^{m}}{1-\mu^{-\varepsilon}}
\end{align*}

which implies that $||\Psi^{\ul{\theta}}_{m+l}-\Psi^{\ul{\theta}}_m|| \rightarrow 0$ as $m \rightarrow \infty$. Further, for any point $z \in \text{Im}(k^{\ul{\theta}}_{n-1})$ we can calculate $ D\Psi^{\ul{\theta}}_n (z) = D\Phi_{\ul{\theta}_n} \cdot Df_ {\left(\theta_{-n, \theta{-n+1}}\right)} (  \Phi_{\ul{\theta}_{n-1}}^{-1} (z) )  \cdot  D\Phi_{\ul{\theta}_{n-1}}^{-1}$. However, by hypothesis, we have that $D\Psi^{\ul{\theta}}_n (c_{\theta_0}) = \text{Id}$ and \[d\left(\Phi_{\ul{\theta}_{n-1}}^{-1} (z),\Phi_{\ul{\theta}_{n-1}}^{-1} (c_{\theta_0})\right) \leq \text{diam} (G^* (\ul{\theta}_{n-1})). \] Then, using that $Df$ is $\varepsilon$ - Hölder, we conclude that
\begin{equation} \label{eq:deriv}
     |D\Psi^{\ul{\theta}}_n (z)-\text{Id}| \leq C_f \cdot |D\Phi_{\ul{\theta}_n}| \cdot |D\Phi_{\ul{\theta}_{n-1}}|^{-1} \cdot  r_{\ul{\theta}_{n-1}}^{\varepsilon}  \leq C_4 \mu^{-n \varepsilon} , 
\end{equation}
since  $|D\Phi_{\ul{\theta}_n}|$ and $|D\Phi_{\ul{\theta}_{n-1}}|$ are comparable (because $D\Phi_{\ul{\theta}_n}=  Df_{\ul{\theta}_{n}}(c_{\theta_0})^{-1}$ and so $|D\Phi_{\ul{\theta}_n}| \cdot |D\Phi_{\ul{\theta}_{n-1}}|^{-1}$ is controlled by $||Dg||$).

Now we can show that $\{||Dk^{\ul{\theta}}_n||\}_{n \geq 1}$ is bounded.  Indeed, $|| Dk^{\ul{\theta}}_n || \leq \prod_{j\geq 2} ^ {n}|| \Psi^{\ul{\theta}}_j || \cdot || Dk^{\ul{\theta}}_1 ||$ implies that:

\begin{align*} 
\log \left( || Dk^{\ul{\theta}}_n ||  \right) & \leq {\sum_{j=2}^{n} \log || \Psi^{\ul{\theta}}_j || } + C_0 \\
& \leq  {\sum_{j=2}^{n} \log |(||\text{Id}||+|| \Psi^{\ul{\theta}}_j - \text{Id}||) } + C_0 \leq   {\sum_{j=2}^{n}  C_4 \mu^{-j \varepsilon}  } + C_0 \\
& \leq \frac{C_4\mu^{-2\varepsilon}+C_0-C_0\mu^{\varepsilon}}{1-\mu^{\varepsilon}} = C_5
\end{align*}

The same argument can be used to show that $|(Dk^{\ul{\theta}}_n)^{-1}|$ is bounded. It also follows that, 
\begin{align*}
||Dk^{\ul{\theta}}_{m+l}-Dk^{\ul{\theta}}_{m}||& \leq \sum_{j=0}^{l-1} ||Dk^{\ul{\theta}}_{m+j+1}-Dk^{\ul{\theta}}_{m+j}|| \leq \sum_{j=0}^{l-1} ||D\Psi^{\ul{\theta}}_{m+j+1}-\text{Id}||\cdot ||Dk^{\ul{\theta}}_{m+j}|| \\
& \leq C_5 \cdot {\sum_{j=0}^{l-1} C_4 \mu^{-(m+j+1) \varepsilon}  } \leq C_6 \cdot  \mu^{-m}
\end{align*}

which shows that $\{k^{\ul{\theta}}_n\}_{n\geq0}$ is a Cauchy sequence also in the  $C^1$ metric, and so it converges to a $C^1$ map $ k^{\ul{\theta}}$. Since $|(Dk^{\ul{\theta}}_n)^{-1}|$ is bounded this also implies that the inverse maps $\{(k^{\ul{\theta}}_n)^{-1}_{n\geq0}\}$ also converge in the $C^1$ metric to the inverse of $k^{\ul{\theta}}$.

In this same manner, we are left with showing that  $ k^{\ul{\theta}}$ is $C^{1+\varepsilon}$. This is true for $k^{\ul{\theta}}_{n}$ for all $n \geq 0$. Indeed, for a given $\ul{\theta} \in \Sigma ^-$ we write:
\[I_{n}(x,y)=|Dk^{\ul{\theta}}_{n}(x)-Dk^{\ul{\theta}}_{n}(y)| < H_n \cdot |x-y|^{\varepsilon}, n \geq 0, x, y \in G^*(\theta_0) .\]

By equation \ref{eq:psi}, equation \ref{eq:deriv} and the fact that $Dk^{\ul{\theta}}_{n}$ are bounded we have that:
\begin{align*}
I_{n}(x,y) & = |D(\Psi_n \circ k^{\ul{\theta}}_{n-1})(x)-D(\Psi_n \circ k^{\ul{\theta}}_{n-1})(y)| \leq \\
           & \leq |D\Psi_n (k^{\ul{\theta}}_{n-1}(x)) \left(Dk^{\ul{\theta}}_{n-1}(x)-Dk^{\ul{\theta}}_{n-1}(y)  \right) | + |\left( D\Psi_n(k^{\ul{\theta}}_{n-1}(x))- \right.\\
           &  \;\;\;    \left.D\Psi_n (k^{\ul{\theta}}_{n-1}(y))\right)Dk^{\ul{\theta}}_{n-1}(y)|\\
           & \leq (1+C_4\cdot\mu^{(-n+1)\varepsilon})\cdot I_{n-1}(x,y) + C_5 \cdot ||D_g|| \cdot |Df_{(\theta_{-n}, \theta_{-n+1})}(\Phi_{\ul{\theta}_{n-1}}^{-1}(x)) \\
           & \;\;\;  -Df_{(\theta_{-n}, \theta_{-n+1})}(\Phi_{\ul{\theta}_{n-1}}^{-1}(y))|\\
           & \leq (1+C_4\cdot\mu^{(-n+1)\varepsilon})\cdot I_{n-1}(x,y) + C_5\cdot ||D_g|| \cdot C_f  \cdot \mu^{(-n+1)\varepsilon} \cdot |x-y|^{\varepsilon}\\
           & \leq \left((1+C_4\cdot\mu^{(-n+1)\varepsilon})\cdot H_{n-1}+  C_5 \cdot ||D_g|| \cdot C_f  \cdot \mu^{(-n+1)\varepsilon} \right)\cdot |x-y|^{\varepsilon},
\end{align*}

which inductively shows that these functions have Hölder continuous derivatives. Additionally, we can choose the Hölder constants satisfying the relation:

\begin{equation}
H_n  \leq (1+C_4\cdot\mu^{(-n+1)\varepsilon}) \cdot H_{n-1}+  C_7 \cdot \mu^{(-n+1)\varepsilon}  
\end{equation}

and then the sequence $\{ H_n \}_{n \geq 1}$ is bounded. Effectively, it is crescent and if $H_{n-1} > 1 $ then $H_n \leq (1+C_4\cdot\mu^{(-n+1)\varepsilon} +  C_7 \cdot \mu^{(-n+1)\varepsilon}) \cdot H_{n-1} \leq (1+C_8\cdot\mu^{(-n+1)\varepsilon}) \cdot H_{n-1}$ and using the same strategy as above we have:

\begin{align*}
    \log H_n & \leq \log H_{n-1}  +  \log (1+C_8\cdot\mu^{(-n+1)\varepsilon}) \\
& \leq \sum_{j=1}^{n-1} \log (1+C_8\cdot\mu^{j\varepsilon}) \leq \sum_{j=1}^{n-1} C_8\cdot\mu^{j\varepsilon} \leq H.
\end{align*}

as stated. 

Finally, for each pair $x,y \in G^*(\theta_0)$ there is $n\geq 0$ such that $||Dk^{\ul{\theta}}_{n}-Dk^{\ul{\theta}}||$ is less than $|x-y|^{\varepsilon}$ so, by triangle inequality we have $|Dk^{\ul{\theta}}(x)-Dk^{\ul{\theta}}(y)| < (H+2) \cdot |x-y|^{\varepsilon}$.

All the constants appearing in the estimatives above depend continuously (actually, they are simple functions) on the $C^1$ norm of $g$ and the Hölder constant and exponent of $Dg$, and so, for any $g'$ sufficiently close to $g$ all of those estimatives would be the same except by a minor pre-fixed error. This implies that the convergence we just shown is uniform not only over $\Sigma^-$ but also on a small neighborhood of $g$ in the Hölder topology. \

Now, for a given $\ul{\theta}\in \Sigma^-$, the norm of $D \left(k^{\ul{\theta}} \circ ({k^{\ul{\theta}}_n)}^{-1}\right)$ is uniformly controlled by the diameter of $G^*(\ul{\theta}_n)$. Indeed,
$$D \left(k^{\ul{\theta}} \circ ({k^{\ul{\theta}}_n)}^{-1}\right)(z)= \lim_{m \rightarrow \infty} D\Psi^{\ul{\theta}}_m ( k^{\ul{\theta}}_{m} \circ ({k^{\ul{\theta}}_n)}^{-1} (z))  \cdot \dots \cdot D\Psi^{\ul{\theta}}_{n+1} (z)$$

and by using the same methods along this proof together with $|D\Psi^{\ul{\theta}}_n (z)-\text{Id}| \leq C_f \cdot ||Dg|| \cdot r_{{\ul{\theta}}_{n-1}}^{\varepsilon} $ we can show that:
$$\log\,\, | D \left(k^{\ul{\theta}} \circ ({k^{\ul{\theta}}_n)}^{-1}\right) | \,\, \leq C_9 \cdot r_{{\ul{\theta}}_n}^{\varepsilon} .$$

Also analogously we have that:

$$ ||k^{\ul{\theta}} - {k^{\ul{\theta}}_n} ||\leq C_{10} \cdot r_{{\ul{\theta}}_n}^{\varepsilon} .$$

So for any $\ul{\theta}_1$ and $\ul{\theta}_2 \in \Sigma^-$ we have  
\[||k^{\ul{\theta}_1} - k^{\ul{\theta}_2} || \leq C_{10} \cdot \text{diam} ( G^* (\ul{\theta}_1 \wedge\ul{\theta}_2 ))^{\varepsilon}  \text{ and } \]
\[||\log(|D\left(k^{\ul{\theta}_1}\circ (k^{\ul{\theta}_2})^{-1}\right)||| \leq  C_9 \cdot \text{diam} ( G^* (\ul{\theta}_1 \wedge\ul{\theta}_2 ))^{\varepsilon}.\]
In this manner, in the metric described in the lemma, the association $\ul{\theta} \rightarrow k^{\ul{\theta}}$ is Hölder continuous as we wished to obtain.

If the map $g$ defining the cantor set is $C^r, r \geq 2$ than the convergence can be taken in the $C^r$ metric. This happens because the composition with affine maps on the definition of $\Psi^{\ul{\theta}}_n$ \enquote{flattens} the derivatives of $f_{\theta_{-n}, \theta_{-n+1}}$. As we have seen above the first order derivatives are close to the identity, or close (in norm) to $1=r_{\ul{\theta}_n}^0$. Analogously, the derivative of $r$ order has norm less than $r_{\ul{\theta}_n}^{r-1}$. This allow us to control the $C^r$ norm of $ k^{\ul{\theta}}_{n}$ by the $C^1$ metric and so the convergence is proved also in this metric.

\end{proof}

As an immediate consequence of the equation  \ref{eq:size} we have the following bounded distortion property:

\begin{corollary}\label{bd}
There is a constant $C>0$ such that for every pair of points $c_1,\, c_2 \in K(a)$ we have
$$C^{-1 }\leq \frac{|Df_{\ul{\theta}_n}(c_1)|}{|Df_{\ul{\theta}_n}(c_2)|} \leq C$$
for all $\ul{\theta}_n \in \Sigma^{fin}$ with $\theta_0=a$.
\end{corollary}

Notice that the limit geometries depend on the choice of the base point $c_{\theta_0}$, because the maps $\Psi^{\ul{\theta}}_n$ depend on it. However, the corollary just above shows for different choices of base point, the norm of the expansion factor of $\Phi_{\ul{\theta}_n}$ is bounded between $|C^{-1} \cdot Df_{\ul{\theta}_n}(c_{\theta_0})|$ and $|C \cdot  Df_{\ul{\theta}_n}(c_{\theta_0})|$ for a fixed choice of $c_{\theta_0}$. Since these maps also send the base point to $0$ we have that different choices of base points $c_1$ and $c_2$ result in different limit geometries that are related by

$$ k^{\ul{\theta}}_1 = A \cdot k^{\ul{\theta}}_2 ,$$

where $A$ is a map in $Aff(\C)$ bounded by some constant $C > 0$. So, less (bounded) affine transformations, the limit geometries do not depend on the base point. Every time we mention the limit geometries of a Cantor set consider that a set of base points has been already fixed. Also we could choose to define $\Phi_{\ul{\theta}_n}$ as the affine map such that $\Phi_{\ul{\theta}_n}( c_{\ul{\theta}_n})=0$, $ |D\Phi_{\ul{\theta}_n}| =\text{diam}( G^*(\ul{\theta}_n)) $ and $D\Phi_{\ul{\theta}_n}(c_{\theta_0})(1,0) \in \R \subset \R^2$ and the resulting limit geometries would only differ from those defined as above by bounded affine transformations. This may be the definition on some other sources.

The bounded distortion property can be improved:

\begin{corollary} There is a constant $C> 0$ such that for every pair of points $x,\,y  \in G^*(\theta_0)$ 

$$ \frac{|Df_{\ul{a}}(x)|}{m(Df_{\ul{a}}(y))}  \leq C \text{ and } C^{-1} \leq \frac{m(Df_{\ul{a}}(x))} {|Df_{\ul{a}}(y)|}$$

for all $ \ul{a}=(a_0,a_1, \dots, a_n) \in \Sigma^{fin}$. The larger $n$ is and the closer $x, y$ are to each other, the closer the ratios $\frac{|Df_{\ul{a}}(x)|}{|Df_{\ul{a}}(y)|}$ and $\frac{m(Df_{\ul{a}}(x))}{m(Df_{\ul{a}}(y))} $ are to $1$.

\end{corollary}

\begin{proof}
    Given any $\ul{\theta}^-$ whose final coincides with the word $\ul{a}$, by the proof of  lemma \ref{limgeo}, we have $\log \,\, | D \left(k^{\ul{\theta}} \circ ({k^{\ul{\theta}}_n)}^{-1}\right)| \,\, \leq C_9 \cdot r_{{\ul{\theta}}_n}^{\varepsilon} $ which implies

\[ \exp(- C \mu^{-n\varepsilon}) \leq m\left(Df_{\ul{a}}(p)(Df_{\ul{a}}(p))^{-1}\right)  \text{ and } |Df_{\ul{a}}(p)(Df_{\ul{a}}(p))^{-1}| \leq \exp(C \mu^{-n\varepsilon})\]

for any $p \in G^*(\theta_0)$. In view of the Hölder continuity of $\ul{\theta} \mapsto k^{\ul{\theta}}$, making $n$ large and $|x-y|$ small the proof is complete.

\end{proof}
We also have the following result:

\begin{corollary}\label{diam}
The diameter of the sets $G^*(\ul{\theta}_n)$ is of order $||Df_{\ul{\theta}_n}||$. 
\end{corollary}

We end this section with the following lemma, that shows continuous dependence of limit geometries on the map $g$ defining the Cantor set.

\begin{lemma} \label{continuous} For any Cantor set $K$ given by a map $g \in \Om_{\Sigma}$ and any $\varepsilon > 0$, there is a small $\delta > 0$ such that for any map $\ti{g} \in U_{K, \delta}$ there is a choice of points $\ti{c}_a \in \ti{G}(a)$ sufficiently close to the points $c_a \in G(a)$ respectively in a manner that the resultant limit geometries satisfy $||\ti{k}^{\ul{\theta}}-{k}^{\ul{\theta}}||_{C^1} < \varepsilon $ for all $\ul{\theta}\in\Sigma^- $ in the largest domain both these maps are defined.
\end{lemma}
\begin{proof}
First we fix $c_a \in \text{int} \, G(a)$. This can be done because of the additional hypothesis on the sets $G(a)$, described at the end of section \ref{first}. By the definition of $U_{K,\delta'}$, we can choose $\ti{c}_a$ such that $H(c_a)= \ti{H}(\ti{c}_a)$, and that implies $|\ti{c}_a-c_a|< \delta$ for every $\delta'$ sufficiently small. Let us analize the limit of $||\ti{k}_n^{\ul{\theta}}-{k}_n^{\ul{\theta}}||_{C^0}$.

Define $x_n =  ||\ti{k}_{n}^{\ul{\theta}}-{k}_{n}^{\ul{\theta}}||_{C^0}$ for $n \in \N$. Given $\varepsilon_0>0$ and $N \in \N$ if the distance between $k_1^{\ul{\theta}}$ and $\ti{k}_1^{\ul{\theta}}$ is small enough then, by continuity, $x_n = ||\ti{k}_n^{\ul{\theta}}-{k}_n^{\ul{\theta}}||_{C^0} < \varepsilon_0$ for all $n \leq N$. We will prove by induction that if $n\geq N$ then there are constants $\ti{C}_1 \text{ and } \ti{C}_2$ such that 
 \begin{equation} \label{eq:series}
      x_{n+1}\leq (1+\ti{C}_1\lambda^n)x_n + \ti{C}_2 \lambda^n
 \end{equation}
 for some $0<\lambda<1$ and for all $n \geq N$ and we can apply $\Psi_{n+1}^{\ul{\theta}}$ to the image of $\ti{k}_n^{\ul{\theta}}$. It is not true that we can always apply $\Psi_{N+1}^{\ul{\theta}}$ to the image of $\ti{k}_N^{\ul{\theta}}$; at least not considering the same domain $\text{Im}(k^{\ul{\theta}_N})$ used in the definition at the beginning of the proof of lemma \ref{limgeo}. However, this domain can be extended to a larger set $V_{\varepsilon_1}(\text{Im}(k^{\ul{\theta}_N}))$, for some $\varepsilon_1>0$, simply because 
 
 $$\Psi^{\ul{\theta}}_n= \Phi_{\ul{\theta}_n} \circ f_ {\left(\theta_{-n}, \theta{-n+1}\right)}\circ \Phi_{\ul{\theta}_{n-1}}^{-1}.$$
is well defined on $\Phi_{\ul{\theta}_{n-1}} (G(\theta_{-n+1})) \supset \Phi_{\ul{\theta}_{n-1}} (G(\ul{\theta}_{-n+1})) = \text{Im}(k^{\ul{\theta_n}})$. So by making $N=n$ and observing that we can make $x_N \leq \varepsilon_0$ the first part of the basis of induction is proved. Now, if we can apply $\Psi_{n+1}^{\ul{\theta}}$ to the image of $\ti{k}_n^{\ul{\theta}}$ for some $n \geq N$ we have, following the notation in the proof of lemma \ref{limgeo}, that:

\begin{align} 
\begin{split} \label{eq:est}
    |\ti{k}_{n+1}^{\ul{\theta}}(x)-{k}_{n+1}^{\ul{\theta}}(x)| 
    & = |(\ti{\Psi}_{n+1}^{\ul{\theta}}\circ\ti{k}_{n}^{\ul{\theta}})(x)-({\Psi}_{n+1}^{\ul{\theta}}\circ{k}_{n}^{\ul{\theta}})(x)|\\ 
    & \leq |\Psi_{n+1}^{\ul{\theta}}(\ti{k}_{n}^{\ul{\theta}}(x))-{\Psi}_{n+1}^{\ul{\theta}}({k}_{n}^{\ul{\theta}}(x))| + |(\ti{\Psi}_{n+1}^{\ul{\theta}}- {\Psi}_{n+1}^{\ul{\theta}})(\ti{k}_{n}^{\ul{\theta}}(x))| \\ 
    & \leq ||D{\Psi}_{n+1}^{\ul{\theta}}||\cdot|\ti{k}_{n}^{\ul{\theta}}(x)-{k}_{n}^{\ul{\theta}}(x)| + || \ti{\Psi}_{n+1}^{\ul{\theta}}- {\Psi}_{n+1}^{\ul{\theta}} || \\ 
    & \leq (1+C_4 \mu^{-n \varepsilon}) \cdot|\ti{k}_{n}^{\ul{\theta}}(x)-{k}_{n}^{\ul{\theta}}(x)| + 2 C_3 \mu^{-n \varepsilon}
\end{split}
\end{align}

which yields a estimative of the type (\ref{eq:series}), and so, the induction basis is complete. However, with basic real analysis, one can show that there are constants $\ti{C}_3$ and $\ti{C}_4$ such that, if (\ref{eq:series}) is true for all $n, \, N\leq n\leq m,\,m \in \N$, then

\[x_m \leq \euler^{\ti{C}_3\lambda^N} (x_N + \ti{C}_4 \lambda^N),\]

and so $x_m, \, m \geq N,$ can be made arbitrarily small by choosing $N$ sufficiently large (and consequently the distance between $k_1^{\ul{\theta}}$ and $\ti{k}_1^{\ul{\theta}}$). This is enough to show that we can apply $\Psi_{m+1}^{\ul{\theta}}$ to the image of $\ti{k}_m^{\ul{\theta}}$ and to do the induction step.

Passing to the limit we show that \[||\ti{k}^{\ul{\theta}}-{k}^{\ul{\theta}}||_{C^0}<\varepsilon\] 
provided that the distance between $k_1^{\ul{\theta}}$ and $\ti{k}_1^{\ul{\theta}}$ is small enough, but this is controlled by the difference between $f_{\theta_{-1},\theta_0}$ and $\ti{f}_{\theta_{-1},\theta_0}$, which can be make small enough by choosing $\delta'>0$ sufficiently small. 

The argument for the $C^{1}$ norm is similar:
\begin{small}
\begin{align*}
    |D\ti{k}_{n+1}^{\ul{\theta}}(x)-D{k}_{n+1}^{\ul{\theta}}(x)| & = |D(\ti{\Psi}_{n+1}^{\ul{\theta}}\circ\ti{k}_{n}^{\ul{\theta}})(x)-D({\Psi}_{n+1}^{\ul{\theta}}\circ{k}_{n}^{\ul{\theta}})(x)|\\ & \leq |D(\ti{\Psi}_{n+1}^{\ul{\theta}}\circ\ti{k}_{n}^{\ul{\theta}})(x)-D(\ti{\Psi}_{n+1}^{\ul{\theta}}\circ{k}_{n}^{\ul{\theta}})(x)| +  |D(\ti{\Psi}_{n+1}^{\ul{\theta}}\circ{k}_{n}^{\ul{\theta}})(x)-D({\Psi}_{n+1}^{\ul{\theta}}\circ{k}_{n}^{\ul{\theta}})(x)| \\ & \leq ||D{\Psi}_{n+1}^{\ul{\theta}}||\cdot|D\ti{k}_{n}^{\ul{\theta}}(x)-D{k}_{n}^{\ul{\theta}}(x)| +  || D\ti{\Psi}_{n+1}^{\ul{\theta}}- D{\Psi}_{n+1}^{\ul{\theta}} || \cdot |Dk_n ^{\ul{\theta}}(x)| \\ & \leq (1+C_4 \mu^{-n \varepsilon}) \cdot|D\ti{k}_{n}^{\ul{\theta}}(x)-D{k}_{n}^{\ul{\theta}}(x)| + 2 {C_4} \euler^{C_5} \mu^{-n \varepsilon}
\end{align*}
\end{small}

and proceeding as above the proof is complete.
\end{proof}

\subsection{Configurations and Renormalizations}\label{attr}

\begin{definition}
Given a dynamically defined conformal Cantor set $K$, described by $ (\mathrm{A},\mathcal{B},\Sigma, g) $ and a piece $G(a),\, a \in \mathrm{A},$ we say that a $C^r, \, r > 1$ diffeomorphism $h: G(a) \to U \subset\C $ is a \textit{configuration} of the piece of Cantor set. 
\end{definition}

In particular, if $h$ is the restriction of a map $A \in Aff(\C)$ to its domain $G(a)$ then we say it is an \emph{affine configuration}. We write $\mathcal{P}(a)$ for the space of all configurations of the piece $G(a)$ equipped with the $C^r$ topology.
 
The space $Aff(\C)$ acts on $\mathcal{P}(a)$ by left composition and we denote the quotient space of this action $\ov{\mathcal{P}}(a)$. We also refer to $\mathcal{P}$ as the union $\bigcup_{a \in \mathrm{A}}\mathcal{P}(a)$ and $\ov{\mathcal{P}}=\bigcup_{a \in \mathrm{A}}\ov{\mathcal{P}}(a)$.

Configurations can be seen as the manner in which the Cantor set is embedded into the complex plane. For example, by using an affine configuration we can rotate, scale and translate a Cantor set that would be fixed in a certain region of the plane. Also, if $h: \bigsqcup_{a \in \mathrm{A}} G(a) \to U \subset \C$ is a $C^r$ diffeomorphism such that $Dh$ is conformal at the Cantor set $K$, then $h(K)$ can be seem as a Cantor set in the previous sense. To see this we need only to consider  new sets $\ti{G}(a)=h(G(a))$ and $\ti{g} = h \circ g \circ h^{-1}$. 

\begin{definition}For any given configuration $h$ we say that $h \circ f_{(\theta_0, \theta_1)}$ is the \textit{renormalization} by $ f_{(\theta_0, \theta_1)}$ of the given configuration and we write the renormalization operator as
\[T_{\theta_0,\theta_1 } : \mathcal{P}(\theta_0) \to \mathcal{P}(\theta_1) \]
\[h \mapsto h \circ f_{(\theta_0,\theta_1)}.\]

\end{definition}
Since this operator commutes with the action of affine maps over $\mathcal{P}$ it is well defined over $\ov{\mathcal{P}}$.

If we apply $n$ consecutive renormalizations, by $f_{(\theta_{0},\theta_1) }, \dots  \,,  f_{(\theta_{n-2}, \theta_{n-1})} $  , $  f_{(\theta_{n-1 }, \theta_n)} $ we end up with $h \circ f_{\ul{\theta}_n}$, $\ul{\theta}_n= (\theta_0,\theta_1, \dots,\theta_n)$. Based on that, we define for any word $\ul{a} \in \Sigma_n, \, \ul{a} = (a_0, \dots, a_n)$ the renormalization operator operator as:

\[T_{\ul{a}}: G(a_0) \to G(a_n)\]
\[h \mapsto h \circ f_{\ul{a}}\]

This construction implies that $T_{\ul{a}}\circ T_{\ul{b}}=T_{\ul{a} \ul{b}}$ for every pair of words $a, \, b \in \Sigma^{fin }$.

Notice that the image of $h \circ f_{\ul{\theta}_n}$ corresponds to the image by $h$ of the set $G(\ul{\theta}_n)$, that is, the configuration of a piece of the $n$-th step in the definition of the Cantor set K, and, as seen in the lemma \ref{limgeo} and its proof, this map is close to $h \circ (\Phi_{\ul{\theta}_n})^{-1} \circ k^{\ul{\theta}}$. This observation indicates that the limit geometries work as an atractor in the space of configurations under the action of renormalizations (less affine transformations). The next two lemmas give a more precise statement of this fact.

First, consider the space $\mathcal{A} = Aff(\C) \times \Sigma^- $. It represents the affine configurations of limit geometries and can be continuously associated with a subset of the space of configurations by:
\[I : Aff(\C) \times \Sigma^- \to \mathcal{P} \]
\[ (A,\theta) \mapsto A \circ k^{\ul{\theta}}\]
Notice that this identification is continuous.

\begin{lemma} \label{afrenor} The action of the renormalization operator over $\mathcal{A}$ is given by: 

\[T_{\theta_0,\theta_1}(A, \ul{\theta}) = (A \circ F^{\ul{\theta}\theta_1},\ul{\theta}\theta_1 )\]

where $F^{\ul{\theta}\theta_1}$ is in $Aff(\C)$.

\end{lemma}

\begin{proof} 
From the lemma \ref{limgeo}, in which we established the existence of limit geometries, we have that:
\begin{align*} k^{\ul{\theta}\theta_1} \circ (k^{\ul{\theta}}\circ f_{\theta_0, \theta_1})^{-1} & = \lim_{n \to \infty} k^{\ul{\theta}\theta_1}_{n+1} \circ (k^{\ul{\theta}}_{n}\circ f_{\theta_0, \theta_1})^{-1}  \\
& = \lim_{n \to \infty} \Phi_{{(\ul{\theta}\theta_1})_{n+1}} \circ f_{{(\ul{\theta}\theta_1)_{n+1}}} \circ (\Phi_{\ul{\theta}_n} \circ f_{{\ul{\theta}_n}}\circ f_{\theta_0, \theta_1})^{-1} \\
& = \lim_{n \to \infty} \Phi_{{(\ul{\theta}\theta_1})_{n+1}} \circ \Phi_{\ul{\theta}_n} ^{-1} 
\end{align*}

which implies that that the last limit exists and in particular belongs to $Aff(\C)$ since this is a closed subset of the space of configurations. So, for any $\ul{\theta} =  ( \dots, \theta_{-1}, \theta_0) \in \Sigma^- $ and $(\theta_0, \theta_1) \in B$ we define $(F^{\ul{\theta} \theta_1}) ^{-1} = \lim_{ n \to \infty} \Phi_{{(\ul{\theta}\theta_1})_{n+1}} \circ \Phi_{\ul{\theta}_n} ^{-1} $ and we have that $F^{\ul{\theta} \theta_1} \circ k^{\ul{\theta}\theta_1} =  k^{\ul{\theta}} \circ f_{\theta_0,\theta_1}$ as we wanted to show.

\end{proof}

\begin{definition}
 For each limit geometry $k^{\ul{\theta}}$,  $\ul{\theta} \in \Sigma^-$, and any configuration $h:G_{\theta_0}\to \C$ the map $h^{\ul{\theta}}: k^{\ul{\theta}}(G_{\theta_0}) \to \C$ is defined as  $h^{\ul{\theta}}= h \circ (k^{\ul{\theta}})^{-1}$, which we call the \textit{perturbation part of $h$ relative to $\ul{\theta}$}.  Also, for each configuration $h \in \mathcal{P}(a)$ we consider the \textit{scaled} version of it as the map $A_h \circ h$, where $A_h \in Aff(\R^2)$ is an affine transformation such that $ A_h \circ h (c_{\theta_0}) = 0  $ and $ D(A_h \circ h) (c_{\theta_0}) = \text{Id}$. 
 \end{definition}
 By definition, $h=h^{\ul{\theta}} \circ k^{\ul{\theta}}$. Also, for example, the scaled version of a limit geometry is the limit geometry itself.
 
 Given a finite word $\ul{a}_n$ of size $n$ with $a_0=a$ for some fixed $a \in \mathbb{A}$ and a configuration $h: G(a_0) \to \C$ we will denote  by $h_n$ the renormalization of $h$ by $\ul{a}_n$ in the next lemma.
 
\begin{lemma}\label{scale}
Let $K$ be a conformal Cantor set and $h \in \mathcal{P}(a_0)$ a configuration of a piece in $K$. Then, for every $n$ large enough and any limit geometry $\ul{\theta} \in \Sigma^-$ with $\theta_0=a_0$, the perturbation part of the scaled version of $h_n$ relative to $\ul{\theta}\ul{a}_n$ converges exponentially to the identity for any $\ul{\theta} \in \Sigma^- $. In other terms, $||A_{h_n} \circ h_n^{\ul{\theta}\ul{a}_n} -\text{Id}  || < C \cdot \text{diam}(G(\ul{a}_n))^{-n\varepsilon}< C \cdot \mu^{-n\varepsilon}$, $C> 0$ a constant depending only on the Cantor set $K$ and the initial configuration $h$.
\end{lemma}
\begin{proof} As seen in lemma \ref{limgeo}, $h^{\ul{\theta}}\circ k^{\ul{\theta}} \circ f_{a_0, a_1} = h^{\ul{\theta}} \circ F^{\ul{\theta}a_1}  \circ k^{\ul{\theta}a_1}  $, and so concatenating all the renormalizations we have:
$$ h^{\ul{\theta}\ul{a}_n}_n = h^{\ul{\theta}} \circ  F^{\ul{\theta}{a}_1} \circ  F^{(\ul{\theta}{a}_1) {a}_2} \circ \cdots  F^{(\ul{\theta}a_1 \dots{a}_{n-1})a_n} ,$$
or equivalently, defining $F^{\ul{\theta}\ul{a}_n}= F^{\ul{\theta}{a}_1} \circ  F^{(\ul{\theta}{a}_1) {a}_2} \circ \cdots  F^{(\ul{\theta}a_1\dots{a}_{n-1})a_n}$, $h^{\ul{\theta}\ul{a}_n}_n= h^{\ul{\theta}} \circ F^{\ul{\theta}\ul{a}_n} $.

The expansion term of $F^{\ul{\theta}\ul{a}_n}$ is equal to $\dfrac{\text{diam}(k^{\ul{\theta}}(G_{\ul{a}_n}))} {\text{diam}(k^{\ul{\theta}\ul{a}_n}(G_{\theta_n}))} $, because of the relation $F^{\ul{\theta}\ul{a}_n} \circ  k^{\ul{\theta}\ul{a}_n}= k^{\ul{\theta}} \circ f_{\ul{a}_n}$ applied to the domain $G(a_n)$. Since any of the maps $k^{\ul{\theta}}$, $\ul{\theta}\in\Sigma^-$, has a uniformly bounded derivative, there is a constant $C > 0$ such that the expansion term of $F^{\ul{\theta}\ul{a}_n}$ is less than  $C \cdot \textrm{diam} (G_{\ul{a}_n})$ and more than  $C^{-1} \cdot \textrm{diam} (G_{\ul{a}_n})$ .

On the other hand, since $ Dk^{\ul{\theta}\ul{a}_n}(c_{a_n}) = \text{Id} \equiv 1 $ and the derivative of $h^{\ul{\theta}}$ is bounded from bellow and above, $m( A_{h_n} )$ and $|A_{h_n}|$ are also controlled by $\textrm{diam} (G_{\ul{a}_n})$ in the same way as $F^{\ul{\theta}\ul{a}_n}$. Now, the domain of interest of $h$ on the relation $h^{\ul{\theta}\ul{a}_n}_n= h^{\ul{\theta}} \circ F^{\ul{\theta}\ul{a}_n} $ is the set $ k^{\ul{\theta}} (G(\ul{a}_n)) $ ($=F^{\ul{\theta}\ul{a}_n} \circ k^{\ul{\theta}\ul{a}_n}(G_{a_n})$) whose size is also controlled by $\textrm{diam} (G_{\ul{a}_n})$. Then, arguing in the same way as in the analysis of the functions $\Psi^{\ul{\theta}}_m$ on lemma \ref{limgeo}, that is, using a relation of the type:

$$ |h(z+h) - (h(z) + Dh(z).h)| < C\cdot |h|^{1+\varepsilon}, $$

and remembering that $A_{h_n} \circ {h_n}^{\ul{\theta}\ul{a}_n}(0) = 0$ and $D(A_{h_n} \circ {h_n}^{\ul{\theta}\ul{a}_n})(0) = \text{Id}$ we have that $$||A_{h_n} \circ {h_n}^{\ul{\theta}\ul{a}_n}-\text{Id}||< C \cdot \text{diam}(G(\ul{a}_n))^{1+\varepsilon}$$ 

The exponential decay of ratio $\mu$ is a consequence of corollary \ref{diam}.

\end{proof}

\subsection{Recurrent compact criterion}\label{recc}

Given a pair of Cantor sets $K \text{ and } K'$ we are interested in finding configurations $h$ and $h'$ such that $h(K)$ intersects $h'(K')$. More importantly we want to find a criterion under which this intersection is stable, that is for small perturbations $\ti{h},\,\ti{h'}, \,\ti{K},\,\ti{K'}$ the sets $\ti{h}(\ti{K})$ and $\ti{h'}(\ti{K'})$ also have a non-empty intersection. 

With these ideas in mind, for any pair of configurations $(h_a, h'_{a'}) \in \mathcal{P}_a \times \mathcal{P}'_{a'} $  we say that it is:

\begin{itemize}
    \item \textit{linked} whenever $h_a(\ov{G(a))}) \cap h'_{a'}(\ov{G(a'))}) \neq \emptyset$.
    \item \textit{intersecting} whenever $h_a(K(a)) \cap h'_{a'}(K'(a')) \neq \emptyset$.
    \item \textit{has stable intersections} whenever $\ti{h}_{a}(\ti{K}(a)) \cap \ti{h'}_{a'}(\ti{K'}(a') \neq \emptyset$ for any pairs of Cantor sets $(\ti{K},\ti{K'}) \in \Om_\Sigma$ in a small neighbourhood of $(K,K')$ and any configuration pair $(\ti{h}_{a},\ti{h'}_{a'})$ that is sufficiently close to $(h_a,h'_{a'})$ in the $C^{1+\varepsilon}$ topology at $G(a) \cap \ti{G}(a)$ and $G(a') \cap \ti{G}'(a')$.
\end{itemize} 
 
It is better to work with $\mathcal{Q}$, the quotient of $\mathcal{P} \times \mathcal{P}'$ by the diagonal action of $Aff(\C)$. An element in $\mathcal{Q}$, represented by a pair $(h,h') $ is called a \emph{relative configuration} or as mentioned sometimes a \emph{relative positioning} of the Cantor sets. Since the action of the affine group preserves the linking, intersecting or stably intersection of a pair of configurations, these notions are defined for relative configurations too. Also, we can define for any pair in $\mathcal{P} \times \mathcal{P}'$ and any pair of words $(\ul{a},\ul{a}') \in \Sigma^{fin} \times {\Sigma'}^{fin}$ a renormalization operator

$${\mathcal{T}}_{\ul{a},\ul{a}'} (h,h') : = (T_{\ul{a}}(h),T_{\ul{a}'}(h')).$$

For the same reasons as above it can also be defined over $\mathcal{Q}$. Also, we can allow one of the words $\ul{a}$ or $\ul{a}'$ to be void. In that case, the operator only acts at the now trivial coordinate, for example $${\mathcal{T}}_{\emptyset,\ul{a}'} (h,h')=(h,T_{\ul{a}'}(h')).$$

Under this context, we have the following:

\begin{lemma} \label{link}A pair of relative configurations $(h_0, h'_0)$ is intersecting if, and only if, there is a relative compact sequence of relative configurations $(h_n,h'_n)$ obtained inductively by applying a renormalization operator on one of the configurations on the pair.

\end{lemma}

\begin{proof}
Indeed, if $h_0(K)$ and $h'_0(K')$ are intersecting at a point $q=h_0(p)=h'_0 (p')$, ($p \in K $ and $p' \in K'$) consider the sequences $H(p) = (a_0, a_1, \dots) \in \Sigma$ and $H'(p)=(a'_0,a'_1, \dots) \in \Sigma'$ where $H$ and $H'$ are the homeomorphism defined in the section \ref{first}. We can construct a sequence of configurations $(h_n, h'_n)$, obtained by successively renormalizing by the functions $f_{a_{i},a_ {i+1}}$ and  $f_{a'_{i},a'_{i+1}}$ chosen in a careful order such that the ratio of diameters of the images of the configurations are bounded away from infinite and zero. 

This can be done as follows. Let $\ul{a}_n = (a_0,a_1, \dots, a_n)$. By lemma \ref{scale} the diameter of $\text{Im}(h_0\circ f_{\ul{a}_n}) = \text{Im}(A_{h_n}^{-1}\circ (A_{h_n}\circ h_n^{\ul{\theta}\ul{a}_n}) \circ k^{\ul{\theta}\ul{a}_n}) $ is comparable to $|A_{h_n}|$ (and $m( A_{h_n} )$), since $\Sigma^-$ is compact. But by the proof of lemma \ref{scale} $|A_{h_n}|$ is comparable to $\text{diam}(G(\ul{a}_n))$, which by equation \ref{eq:size} of lemma \ref{limgeo} and corollary \ref{bd} is comparable to $\text{diam}(G(\ul{a}_{n-1}))$. Hence, diameter of $\text{Im}(h_0\circ f_{\ul{a}_n})$ is comparable to the diameter of $\text{Im}(h_0\circ f_{\ul{a}_{n-1}})$ independently of $n$ and the same analysis is valid for $K'$ and $h'_0$, making the construction above possible.

Finally, such pairs of configurations are always intersecting, since the point $q$ belongs to both their images, and so, when seen in $\mathcal{Q}$ this sequence is relatively compact.

On the other hand, if such a relative compact sequence exists, choosing points $p_n \in h_n(K) \subset h_0(K)$ and  $p'_n \in h'_n(K') \subset h'_0(K')$ we have that $\lim_{n \to \infty} p_n = p = \lim_{n \to \infty} p'_n $ and then $p \in h_0(K) \cap h'_0(K') \neq \emptyset $ as we wanted to show.

\end{proof}

This lemma is really important in finding a criterion for stable intersection in Cantor sets. To do it, we will work with the space of relative affine configurations of limit geometries, that is, $\mathcal{C}$,  the quotient by the action of the affine group on the left of $\mathcal{A} \times \mathcal{A'} $. The concepts above were well defined for pairs of affine configurations of limit geometries, and again, since they are invariant by the action of $Aff(\C)$ they are also defined for relative configurations in $\mathcal{C}$. Also, since the renormalization operator acts by multiplication on the right on $(A,\ul{\theta})$ its action commutes with the multiplication on the left by affine transformations and so it can be naturally defined on $\mathcal{C}$.

\begin{definition}[Recurrent compact criterion] Let $\mathcal{L}$ be a compact set in $\mathcal{C}$. We say that $\mathcal{L}$ is \emph{recurrent}, if for any relative affine configuration of limit geometries $v \in \mathcal{L}$, there are words $\ul{a}$, $\ul{a}'$ such that $u=\mathcal{T}_{\ul{a},\ul{a'}} (v) $, satisfies $u \in \text{int }\mathcal{L}$.

If such a renormalization can be done using words $\ul{a}$ and $\ul{a}'$ such that their total size combined is equal to one, we say that such a set is \textit{immediately recurrent}.

\end{definition}

\begin{theorem*b} The following properties are true:
\begin{enumerate}
    \item Every recurrent compact set is contained on an immediately recurrent compact set.
    \item Given a recurrent compact set $\mathcal{L}$ (resp. immediately recurrent) for $g$, $g'$, for any $\ti{g}$,$\ti{g}'$ in a small neighbourhood of  $(g,g') \in \Om_{\Sigma} \times \Om_{{\Sigma}'}$ we can choose points $\ti{c}_a \in \ti{G}(a) \subset \ti{K}$ and  $\ti{c}_{a'} \in \ti{G}(a') \subset \ti{K}'$ respectively close to the pre-fixed $c_a$ and $c_{a'}$ in a manner that $\mathcal{L}$ is also a recurrent compact set for $\ti{g}$ and $\ti{g}'$.
    \item Any relative configuration contained in a recurrent compact set has stable intersections.
\end{enumerate}

\end{theorem*b} 

\begin{proof}

(1) We remember that $\mathcal{A}$ is a metric space, so for every point $v \in \mathcal{L}$, the recurrent compact set, there is a closed ball around $v$ that is carried by a renormalization (given by a pair of words $\ul{a}^{v}, \, {\ul{a}'}^{v}$) into the interior of $\mathcal{L}$. Since this set is compact, there is a finite number $N$, compact sets (balls) $\mathcal{L}^{i}$ and associated pair of words $(\ul{a}^{i},{\ul{a}'}^{i})$ for $1 \leq i \leq N$ such that $\mathcal{L}^i$ is carried into the interior of $\mathcal{L}$ by the renormalizations associated to the pair $(\ul{a}^{i},{\ul{a}'}^{i})$. Now considering for every such pair, all the pairs of words $(\ul{b}^{i,j},{\ul{b}'}^{i,j})$ that are contained on $(\ul{a}^{i},{\ul{a}'}^{i})$. We construct an immediately recurrent Cantor set  $\mathcal{L}'$ on the following way. 
    
First, we choose one of the pairs $(\ul{b}^{i,j_1},{\ul{b}'}^{i,j_1})$ with total size one and construct a compact set $\mathcal{L}^{i,j_1}$ such that $T_{(\ul{b}^{i,j_1},{\ul{b}'}^{i,j_1})} \mathcal{L}^i \subset \text{int } \mathcal{L}^{i,j_1}$ and $T_{(\ul{b}^{i,\ov{j_1}},{\ul{b}'}^{i,\ov{j_1}})} \mathcal{L}^{i,j_1} \subset \text{int } \mathcal{L} $, where $(\ul{b}^{i,\ov{j_1}},{\ul{b}'}^{i,\ov{j_1}})$ is the word pair that need to be concatenated into $(\ul{b}^{i,j_1},{\ul{b}'}^{i,j_1})$ to result in $(\ul{a}^{i},{\ul{a}'}^{i})$. This can be done in the same manner the sets $\mathcal{L}^i$ were constructed just above.
        
Then we inductively construct a sequence of compact sets $\mathcal{L}^{i,j_k}$ $k=1, \dots, m$ such that for each $\mathcal{L}^{i,j_k}$ there is a renormalization by words of total size one that carries it into $\text{int } \mathcal{L}^{i,j_{k+1}}$ for $k= 1, \dots, m-1$ and carries $\mathcal{L}^{i,j_m}$ into $\text{int } \mathcal{L}$. Then, taking $\mathcal{L'} = \bigcup_{i,j_k}  \mathcal{L}^{i,j_k}$ we have an immediately recurrent compact set. 
    
(2) Using the decomposition $\mathcal{L} = \cup^{N}_{i=1} \mathcal{L}_i$, described in the previous argument, it is enough to show that for any $\varepsilon>0$ there is $\delta$ such that for any pair $(\ti{g}, \ti{g}')$ in $U_{K,\delta} \times U_{K',\delta} $ the renormalizations operators associated to the words $\ul{a}^{i} \text{ and } {\ul{a}'}^{i}$, $i=1, \dots, N$, denoted by  $\ti{\mathcal{T}}_{\ul{a}^{i},{\ul{a}'}^{i}}=(\ti{T}_{\ul{a}^{i}},\ti{T}_{{\ul{a}'}^{i}})$ satisfy $ ||\ti{\mathcal{T}}-\mathcal{T}|| < \varepsilon$. 

These operators are obtained by composition of a finite number of  operators arising from pairs of words of total size one, and so we need only to that for any $\varepsilon'>0$ a $\delta$ can be find such that $|T_{a,b}-\ti{T}_{a,b}|<\varepsilon'$ and  $|T_{a',b'}-\ti{T}_{a',b'}|$ for all pairs $(a,b) \in B$ and $(a',b') \in B'$, or precisely, that for any $\ul{\theta} \in \Sigma^-$ and $\theta_1 \in \mathrm{A}$, $|F^{\ul{\theta}\theta_1} -\ti{F}^{\ul{\theta}\theta_1}  | < \varepsilon'$  and its analogous for $K' \text{ and } \ti{K}'$. However as seen in lemma \ref{afrenor} $F^{\ul{\theta}\theta_1}=k^{\ul{\theta}} \circ f_{\theta_0,\theta_1} \circ (k^{\ul{\theta}\theta_1})^{-1}$ an we need only to show that the values of all the function and their derivatives above at a fixed point $x$ do not change much when considering its $ \ti{K} $ version. This is was done in lemma \ref{continuous}.
    
(3) Given a recurrent compact set $\mathcal{L}$ relative to a pair of Cantor sets $K$ and $K'$ its image under
\[I : \mathcal{C} \to \mathcal{Q} \]
\[ [(A,\theta),(A',\theta')] \mapsto [A \circ k^{\ul{\theta}},A' \circ k^{\ul{\theta}'}]\]
is also a compact set, because this association is continuous. In what follows, whenever we work with the set $\mathcal{L}$ in the context of $\mathcal{Q}$ we are referencing the set $I(\mathcal{L})$. In this sense, following lemma \ref{link}, any pair $(A\circ k^{\ul{\theta}},{A'}\circ {k}^{\ul{\theta}'})$ representing a relative affine configuration of limit geometries $v \in \mathcal{L}$ is intersecting. In light of the previous item, this implies that $A \circ \ti{k}^{\ul{\theta}}$ and $A' \circ \ti{{k}}^{\ul{\theta}'}$ also represent intersecting configurations for a pair of Cantor sets $(K',\ti{K'})$ sufficiently close to $(K,K')$.
    
Thus, it is enough to show that, for any pair of Cantor sets $(K,K')$ that has a recurrent compact set $\mathcal{L}$ and a configuration pair $v \in \mathcal{L}$ represented by $[(A,\ul{\theta}),(A',\ul{\theta}')]$, if $h : \text{Im}(A \circ k^{\ul{\theta}}) \to \C$ and $h' : \text{Im}(A \circ k^{\ul{\theta}'}) \to \C$ are embeddings $C^{1+\varepsilon}$ close to the identity, then $h \circ A \circ k^{\ul{\theta}}$ and $h' \circ A' \circ k^{\ul{\theta}'}$ are also intersecting. 
    
Again, it is a simple observation from real analysis that:
$$ |h(z+\delta) -(h(z)+Dh(z) \cdot \delta)|<C \cdot |\delta|^{1+\varepsilon}; \, z, \, z+\delta \in \text{Dom}(h)$$
for some constant $C> 0$ that depends only on the map $h$. Without loss of generality we can assume the same for the map $h'$. Also again, as seen in the analysis of the functions $\Psi^{\ul{\theta}_n}$, if $\mathfrak{h} = A \circ h|_{X} \circ B$, $A, \, B \in Aff(\R^2)$ are maps chosen such that $\mathfrak{h}$ has a fixed point $p$ with $D\mathfrak{h}(p)= \text{Id}$ it is $C \cdot \text{diam}(X)^{\varepsilon}$ - $C^1$ close to the identity, and the analogous result is valid for $h'$. 

Now consider a decomposition of the recurrent compact set $\mathcal{L} = \cup \mathcal{L}^i$ as done in item 1 and fix a set of renormalizations $T_{\ul{a}^i,{\ul{a}'}^i}$ that carries each $\mathcal{L}^i$ to the interior of $\mathcal{L}$. We can find a $\delta>0$ such that the distance of any $T_{\ul{a}^i,{\ul{a}'}^i}(v),\, v \in  \mathcal{L}$, to the boundary of $\mathcal{L}$ is bigger than $\delta$.
    
We can construct a sequence of pairs of relative configurations $[h_n,h'_n]$, $n \geq 0$, obtained by inductively applying one of the renormalization operators $T_{\ul{a}^i,{\ul{a}'}^i}$ to the previous element, that starts with $[h \circ A \circ k^{\ul{\theta}}, h' \circ A' \circ k^{\ul{\theta}'})]$ and is relatively compact. The idea is to use the series of renormalizations that worked for $[A\circ k^{\ul{\theta}},A'\circ k^{\ul{\theta}'} ]$ and do small adaptations along the way.  We do as follows.
    
First, we obtain $[h_1,h'_1]$ from $[h_0,h'_0]$ by one of the renormalizations above that carries $[A\circ k^{\ul{\theta}},A'\circ k^{\ul{\theta}'} ]$ to the interior of $\mathcal{L}$. We choose to look at any of the pairs $[h_n,h'_n]$ by its scaled representant $(A_{h_n} \circ h_n, A_{h_n}\circ h'_n )$. Also, we write $A_m = A \circ F^{\ul{\theta}\ul{a}_m} $ for the words $\ul{a}_m$, $m \in \N$ appearing in the argument bellow (and also analogously for $A'$).  In this manner we can write:
    
\begin{align*}
        [h_1,h'_1] & =[A_{h_1} \circ h \circ A_{n1} \circ k^{\ul{\theta}\ul{\theta}_{n1}} ,A_{h_1} \circ h' \circ A'_{n'1} \circ k^{\ul{\theta}'\ul{\theta}'_{n'1}} ]\\ 
        & = [\mathfrak{h}_1 \circ k^{\ul{\theta}\ul{\theta}_{n1}}, \mathfrak{h}'_1 \circ B_1 \circ  k^{\ul{\theta}'\ul{\theta}'_{n'1}} ]
\end{align*}
    
in which $\mathfrak{h}_1 = A_{h_1} \circ h \circ A_{n1}$, $\mathfrak{h}'_1= A_{h_1} \circ h' \circ A'_{n'1} \circ B_1^{-1}$, and $B_1 \in Aff(\C)$ is constructed below.
    
First, we decompose $ T^1 = DA_{h_1}\cdot Dh'(A'_{n'1}(0)) = P^1 \cdot B^1$; $B^1, \, P^1 \in \text{GL}_2(\R)$ but $B^1$ conformal. This decomposition can be done in this systematic way:
    
\begin{itemize}
        \item If $T^1$ is conformal, then $B^1=T^1$.
        \item If not, then if  $\lambda_1$ and $\lambda_2$ are the eigenvalues of $T^1$ (with repetition), we choose $B^1= \sqrt{\lambda_1\lambda_2}\cdot\text{Id}$.
\end{itemize}
    
    Under this notation $B_1 \in Aff(\C)$ is equal to
    
    $$ B^1 \cdot DA'_{n'1} \cdot (z-(A_{h_1} \circ h' \circ A'_{n'1})(0)) + (A_{h_1} \circ h' \circ A'_{n'1})(0) $$
    
    
    Notice that if $h$ and $h'$ are chosen sufficiently $C^{1}$ close to the identity $( k^{\ul{\theta}\ul{\theta}_{n1}}, B_1 \circ  k^{\ul{\theta}'\ul{\theta}'_{n'1}} ))$ represents a relative pair of configurations that is still in $\mathcal{L}$. In fact, by hypothesis $(k^{\ul{\theta}\ul{\theta}_{n1}}, A_{A_{n1} \circ k^{\ul{\theta}\ul{\theta}_{n1}}} \circ A'_{n'1} \circ k^{\ul{\theta}'\ul{\theta}'_{n'1}} ) \in \text{int}(\mathcal{L})$, and then, since $A_{h_1}$ is close to $A_{A_{n1} \circ k^{\ul{\theta}\ul{\theta}_{n1}}}$,  $B_1$ can be shown to be close to $A_{A_{n1} \circ k^{\ul{\theta}\ul{\theta}_{n1}}} \circ A'_{n'1} $.
    
    In light of a previous observation, the definition of $A_{h_1}$implies that $$||\mathfrak{h}_1-\text{Id}||_{C^1} < C \cdot \text{diam}(\text{Im}( A_{n1} \circ k^{\ul{\theta}\ul{\theta}_{n1}}))^{\varepsilon}.$$
    
    Also, in a similar manner, if we construct $P_1 \in Aff(\R^2)$ as
    \[ P_1 = P^1  \cdot (z-(A_{h_1} \circ h' \circ A'_{n'1})(0)) + (A_{h_1} \circ h' \circ A'_{n'1})(0), \]
    
    then the definition of $B_1$ implies that:
    
    $$ ||\mathfrak{h}'_1-P_1||_{C^1} < C \cdot \text{diam}(\text{Im}( A'_{n'1} \circ (B_1)^{-1} \circ k^{\ul{\theta}'\ul{\theta}'_{n'1}}))^{\varepsilon} $$
    
    Inductively, we have
    \[[h_{m-1},h'_{m-1}] = [\mathfrak{h}_{m-1} \circ k^{\ul{\theta}\ul{\theta}_{n(m-1)}}, \mathfrak{h}'_{m-1} \circ B_{m-1} \circ  k^{\ul{\theta}'\ul{\theta}'_{n'(m-1)}} ]\]
    with $B_{m-1} \in Aff(\C)$ such that 
    \begin{align*}
    [k^{\ul{\theta}\ul{\theta}_{n(m-1)}}, B_{m-1} \circ  k^{\ul{\theta}'\ul{\theta}'_{n'(m-1)}} ] \in \text{int}\,\mathcal{L},\\
    ||\mathfrak{h}_{m-1}-\text{Id}||_{C^1} < C \cdot \text{diam}(X_{m-1})^{\varepsilon}, \text{and}\\
     ||\mathfrak{h}'_{m-1}-P_{m-1}||_{C^1} < C \cdot \text{diam}(X'_{m-1})^{\varepsilon}
    \end{align*}
    
    where $P_{m-1}\in Aff(\R^2)$ is an affine map close to identity. To obtain $[h_m,h'_m]$ we now renormalize by one of the operators $T_{\ul{a}^i,{\ul{a}'}^i} $ that carries $( k^{\ul{\theta}\ul{\theta}_{n(m-1)}},  B_{m-1} \circ  k^{\ul{\theta}'\ul{\theta}'_{n'(m-1)}} )$ into the interior of $\mathcal{L}$. The result can analogously be represented as
    
    $$ (\mathfrak{h}_m \circ k^{\ul{\theta}\ul{\theta}_{n(m)}}, \mathfrak{h}'_m \circ B_m \circ  k^{\ul{\theta}'\ul{\theta}'_{n'(m)}} ) $$
     by making the same construction just above to create $B_m$ in $Aff(\C)$ and $P_m$ in $Aff(\R^2)$ as
     \begin{footnotesize}
     \begin{align*}
         B_m= B^{m} \cdot DA'_{n'(m)} \cdot (z-(A_{h_{m}} \circ h' \circ A'_{n'(m)})(0)) + (A_{h_{m}} \circ h' \circ A'_{n'(m)})(0)\\
          P_m= P^m  \cdot (z-(A_{h_m} \circ h' \circ A'_{n'(m)})(0)) + (A_{h_m} \circ h' \circ A'_{n'(m)})(0),
        \end{align*}
        \end{footnotesize}
 with $P^m$ and $B^m$ defined as were $P^1$ and $B^1$.         

By the same argument we can show that:
     $$||\mathfrak{h}_m-\text{Id}||_{C^1} < C \cdot \text{diam}(X_m)^{\varepsilon}, \text{and}$$
    $$ ||\mathfrak{h}'_m-P_m||_{C^1} < C \cdot \text{diam}(X'_m)^{\varepsilon} .$$
    
    where $X_m = \text{Im}( A_{n(m)} \circ k^{\ul{\theta}\ul{\theta}_{n(m)}})$ and $X'_m = \text{Im}( A'_{n'(m)} \circ (B_m)^{-1} \circ k^{\ul{\theta}'\ul{\theta}'_{n'(m)}})$. Observing that the diameters of these sets converge exponentially to zero and that the affine maps $P_m$ are within a controlled distance of the identity (to see this observe that $||P_m-\text{Id}||$ depends only on the distortion of $Dh$ and $Dh'$ at some specific points) we have the sequence of renormalizations are indeed relatively compact, in fact, in a bounded (small) neighbourhood of $\mathcal{L}$, finishing the proof.
    
    However, there is a little techinicality on doing the induction step: the estimative we have on $\mathfrak{h}_{m-1}$ and $\mathfrak{h'}_{m-1}$ does not guarantee that the maps involved in the definition of $B_{m}$ are going to be sufficiently close to the identity. It only guarantees this for $m>m_0$, for a large fixed $m_0$, since at least one of the diameters of $X_m$ or $X'_m$ diminishes by at least a factor $r<1$. Even so, all this only works if the induction steps work for all the $m_0$ steps that came before this one. Nonetheless, we can always guarantee that $||\mathfrak{h}_m-Id||_{C^1} < R \cdot ||\mathfrak{h}_{m-1}-Id||_{C^1}$ for some constant $R>0$ and analogously for the $\mathfrak{h}'_m$ and $\mathfrak{h}'_{m-1}$. So, if the initial distance to the identity of $h$ and $h'$ is sufficiently small, we can work the first $m_0$ steps of the induction with this new estimative and proceed as previously for $m>m_0$.

\end{proof}
\section{Constructing a compact reccurent set for Buzzard's example.}\label{last}

In the article \cite{buzz}, Buzzard found an open set $U \subset Aut(\C^2)$ with a residual subset $\mathcal{N} \subset U$ with coexistence of infinitely many sinks, thus establishing the existence of Newhouse phenomenon on the two dimensional complex context. The strategy was very similar to the one by Newhouse in his works \cite{n_1}, \cite{n_2} and \cite{n_3}. Consider the example \ref{shoe} in section \ref{3}.

Checking the section 5 of \cite{buzz}, we find a very favourable construction of a tangency between $W^s_F(0)$ and $W^u_F(0)$; the disk of tangencies $D_T$ is equal to a small vertical plane $\{q\} \times \rho_2 \cdot \mathbbm{D}$ and, choosing a suitable parametrization of $D_T$  we can assume that the projections $\Pi_s$ and $\Pi_u$ along the stable and unstable foliations from $W^u_{F,\text{loc}}$ and $W^s_{F,\text{loc}}$ to $D_T$  are the identity (considering also the obvious inclusion of such sets to $\C$).

Moreover, we remember that, as already discussed in section \ref{3}, for any $G \in Aut(\C^2)$ such that $||G|_{K_1}-F||$ is sufficiently small, both unstable and stable foliations are also defined and can be taken $C^r,\,r>1,$ very close to the vertical and horizontal foliations. That means, denoting by $p_G$ the continuation of the fixed point $0$ for $F$, there are continuations $W^{u,s}_{G,\text{loc}}(p_G)$, parametrized by $a^u(w) = (\alpha^u(w),w), \, w \in S(0;c_0)$ and $a^s(z)=(z,\alpha^s(z)), z \in S(0;c_0)$ respectively,  with $\alpha^s \text{ and } \alpha^u$  very close to zero, such that the sets $(a^u)^{-1}( W^{u}_{G,\text{loc}}(p_G \cap \Lambda_G)) = K^s_G$ and $(a^s)^{-1}(W^{s}_{G,\text{loc}}(p_G) \cap \Lambda_G)=K^u_G$ are Cantor sets very close to $K$ in the topology we consider for $\Om_{P^{\N}}$. Further, the line of tangencies $D^G_T$ is also well-defined and can be parametrized by something close to the parametrization of $D_T$. Therefore, the projections $\Pi_s$ from $W^u_{G,\text{loc}}$ to $D^G_T$ and the projection $\Pi_u$ from $W^s_{G,\text{loc}}$ to $D^G_T$ can be seen, under these parametrizations, as diffeomorphisms $h^s$ and $h^u$ very close to the identity.

The existence of a tangency between between $W^s(\Lambda_G)$ and $W^u(\Lambda_G)$ corresponds to a intersection between $h^s(K^s_G)$ and $h^u(K^u_G)$. Consequently, if we can show that the configuration pair $(\text{Id},\text{Id})$ has stable intersections for the pair $(K,K)$ of conformal Cantor sets, then, for every $G \in Aut(\C^2)$ such that $||G|_{K_1}-F||$ is sufficiently small there is a homoclinic tangency at $G$. We show that this is the case:

\begin{theorem*c}
There is $\delta$ sufficiently small for which the pair of Cantor sets $(K,K)$ defined above has a recurrent compact set of affine configurations of limit geometries $\mathcal{L}$ such that $[\text{Id},\text{Id}] \in \mathcal{L}$.
\end{theorem*c}

\begin{proof}
The first observation is that the maps defining $K$
\begin{align*}
    g_a: S(a;c_0) & \to S(0; 3)\\
            z & \mapsto \frac{3}{c_1}(z-a)
\end{align*}
 are all affine. Hence, if ${\ul{\theta}} \in (P^{\N}) ^-$ has $\theta^0=a$ then $k^{\ul{\theta}}$ is an affine transformation with derivative $\text{Id} \equiv 1 \in \C$ that carry a base point to $0$. So, choosing for any of the pieces $S(a;c_0)$ the base point $c_a=a$ we have $k^{\ul{\theta}}(z)=z-a$, whose image is always the set $S(0;c_0)$. 
 
 It is also easy to verify that, under our notation, for any $a, \, b \in P$, $ f_{a,b}(z)=\frac{c_1}{3}z+a $. We can them verify that the action of the renormalization operators is described by
 $$F^{\ul{\theta}ab}(z)= \frac{c_1}{3}(z+b).$$
 
As already discussed, we denote every configuration pair $[h,h'] \in \mathcal{Q}$ by its representative that is scaled in the first coordinate $(A_h \circ h, A_h \circ h')$. Similarly, any configuration pair $[(A,{\ul{\theta}}),(A',{\ul{\theta'}})] \in \mathcal{C} $  will be represented by the \enquote{triple} $(\ul{\theta},\ul{\theta}',\text{Id},A^{-1}\circ A)$. However, proceeding by algebraically calculating the renormalization operator under this identification makes it hard to construct a recurrent compact set, so we also choose a geometric interpretation. For this we may, identify any map $B \in Aff(\C)$ with  $B(S(0;c_0))$, which is a square embedded on $\C$, considering the orientation of its vertices. The figure below exemplifies this idea for our identification.

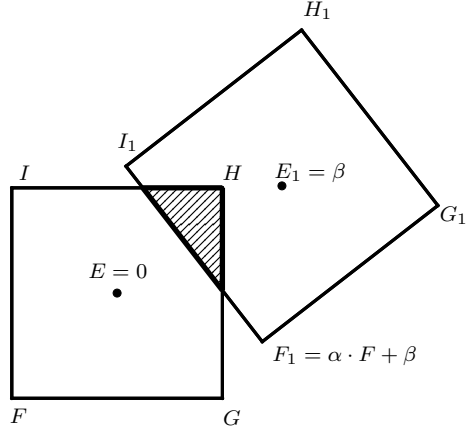
\begin{figure}[ht]
\centering
\begin{tikzpicture}[x=0.35cm,y=0.35cm]
\draw[line width=1.2pt] (-4.57917725411484,-3.7481838025587643) -- (3.4208227458851583,-3.7481838025587643) -- (3.4208227458851587,4.251816197441233) -- (-4.579177254114839,4.251816197441235) -- cycle;
\draw[line width=1.2pt] (4.935430204787627,-1.5968121494266012) -- (11.605253081178602,3.5797667994141577) -- (6.428674132337847,10.249589675805133) -- (-0.24114874405313014,5.073010726964377) -- cycle;
\draw[pattern=north east lines,pattern color=black] (0.3961962639349816,4.251816197441233) -- (3.4208227458851587,4.251816197441233) -- (3.4208227458851587,0.3547013072361964) -- cycle;
\draw [line width=1.2pt] (-4.57917725411484,-3.7481838025587643)-- (3.4208227458851583,-3.7481838025587643);
\draw [line width=1.2pt] (3.4208227458851583,-3.7481838025587643)-- (3.4208227458851587,4.251816197441233);
\draw [line width=1.2pt] (3.4208227458851587,4.251816197441233)-- (-4.579177254114839,4.251816197441235);
\draw [line width=1.2pt] (-4.579177254114839,4.251816197441235)-- (-4.57917725411484,-3.7481838025587643);
\draw [line width=1.2pt] (4.935430204787627,-1.5968121494266012)-- (11.605253081178602,3.5797667994141577);
\draw [line width=1.2pt] (11.605253081178602,3.5797667994141577)-- (6.428674132337847,10.249589675805133);
\draw [line width=1.2pt] (6.428674132337847,10.249589675805133)-- (-0.24114874405313014,5.073010726964377);
\draw [line width=1.2pt] (-0.24114874405313014,5.073010726964377)-- (4.935430204787627,-1.5968121494266012);
\draw [line width=2pt] (0.3961962639349816,4.251816197441233)-- (3.4208227458851587,4.251816197441233);
\draw [line width=2pt] (3.4208227458851587,4.251816197441233)-- (3.4208227458851587,0.3547013072361964);
\draw [line width=2pt] (3.4208227458851587,0.3547013072361964)-- (0.3961962639349816,4.251816197441233);
\begin{scriptsize}
\draw [fill=black] (-4.57917725411484,-3.7481838025587643) circle (0.5pt);
\draw[color=black] (-4.327379720753465,-4.429987134830876) node {$F$};
\draw [fill=black] (3.4208227458851583,-3.7481838025587643) circle (0.5pt);
\draw[color=black] (3.7811018368543063,-4.480350374319124) node {$G$};
\draw [fill=black] (3.4208227458851587,4.251816197441233) circle (0.5pt);
\draw[color=black] (3.80628345659843,4.811667311262469) node {$H$};
\draw [fill=black] (-4.579177254114839,4.251816197441235) circle (0.5pt);
\draw[color=black] (-4.100745143056353,4.8368489310065925) node {$I$};
\draw [fill=black] (4.935430204787627,-1.5968121494266012) circle (0.5pt);
\draw[color=black] (8.068996838687076,-2.114619209302736) node {$F_1=\alpha \cdot F + \beta$};
\draw [fill=black] (11.605253081178602,3.5797667994141577) circle (0.5pt);
\draw[color=black] (12.216944451135685,3.200043647638561) node {$G_1$};
\draw [fill=black] (6.428674132337847,10.249589675805133) circle (0.5pt);
\draw[color=black] (7.044060363376122,11.006345768316862) node {$H_1$};
\draw [fill=black] (-0.24114874405313014,5.073010726964377) circle (0.5pt);
\draw[color=black] (-0.1724124629730852,5.995203439236276) node {$I_1$};
\draw [fill=black] (-0.5791772541148399,0.25181619744123473) circle (1.5pt);
\draw[color=black] (-0.550136759134938,1.1099692088763058) node {$E = 0$};
\draw [fill=black] (5.682052168562737,4.326388763189267) circle (1.5pt);
\draw[color=black] (6.7525329666608815,4.8368489310065925) node {$E_1 = \beta$};
\end{scriptsize}
\end{tikzpicture}
\caption{The first square represents the configuration obtained from Id. The second represents the one from $A^{-1}\circ A(z)=\alpha\cdot z+\beta$. We will measure the distance of the configuration to the identity by the marked area the squares above have in common.} 
\label{fig:M1}
\end{figure}
The square $FGHI$ represents $\text{Id}(S(0;c_0))$ with $E=0$ as its center. The square $F_1G_1H_1I_1$ represents the image of $S(0;c_0)$ by $A^{-1}\circ A$. If $A^{-1}\circ A=\alpha\cdot z+ \beta, \, \alpha,\,\beta \in \C$, then the vector $\Vec{EE_1}$ represents $\beta$ and $\alpha =\frac{ A_1-b}{A}$, when $A,\,A_1,\,\beta$ are seen as complex numbers. In this figure, one can easily see that $\alpha=R\cdot \exp{\mi \Om}$, with $R>1$ and $\Om \in (0,\sfrac{\pi}{4})$. 

For each $\kappa \in (0,1]$ and each complex number $\alpha$ define $X^\kappa_\alpha$ as the set of all  $B \in Aff(\C)$ that are equal to $\alpha\cdot z + \beta, \text{ for some } \beta\in \C$, such that  the area of $  S(0;c_0) \cap B(S(0;c_0) ) $ is at least $ \kappa \cdot c_0^2 \cdot (1+|\alpha|^2)$, meaning that their intersection has an area of at least a small percentage of the sum of their areas. For $\kappa=0$ we write $X^\kappa_\alpha$ for the set of all $B \in Aff(\C) \text{ that are equal to }  \alpha\cdot z + \beta, \text{ for some } \beta\in \C$, such that $  S(0;c_0) \cap B(S(0;c_0) ) \neq \emptyset$. For example, if we consider the affine map identified with the figure just above, it is true that, for $\kappa$ very close to $0$, it is in $ X^\kappa_\alpha$. Also for any real number $c \in (0,1]$ define the set $R_{c}=\{z \in \C;\, {c}^{\sfrac{1}{2}} \leq |z| \leq {c}^{-\sfrac{1}{2}}\}$.

Let $\frac{1}{4}< \frac{c_1}{3} \leq c < 1$. If $\kappa_1 < \frac{c}{36(1+c)}$ then, for any $A \in X^0_\alpha \setminus X^{\kappa_1}_\alpha  $ with $\alpha \in R_{c}$ the intersection $S(0;c_0) \cap A(S(0;c_0))$ is contained in one of the four strips $S_1=\{z;\,\text{Re}(z) > \frac{c_0}{3}\} \cap S(0;c_0)$, $S_2=\{z;\,\text{Re}(z) < -\frac{c_0}{3}\} \cap S(0;c_0)$, $S_3=\{z;\,\text{Im}(z) > \frac{c_0}{3}\} \cap S(0;c_0)$ and $S_4=\{z;\,\text{Im}(z) < -\frac{c_0}{3}\} \cap S(0;c_0)$. It is also contained in one of the for strips $A(S_i), \, i=1,2,3,4.$ This estimative comes form the fact that if that would not be the case, then $\sfrac{1}{3} \cdot S(0;c_0) \cap A(S(0;c_0)) $ or $\ S(0;c_0) \cap A(\sfrac{1}{3} \cdot S(0;c_0)) $ would be non-empty. But in any of these cases, the area of $S(0;c_0) \cap A(S(0;c_0)) $ would be larger than $\frac{c}{36} $. 

Now, let $0<\kappa<\kappa_1 $. We show that if $c_1$ is really close to $1$ then for any $v \in \mathcal{C}$ identified by $(\ul{\theta},\ul{\theta}',\text{Id},A)$ with $A \in X^\kappa_\alpha \setminus X^{\kappa_1}_\alpha $ and $\alpha \in R_{\frac{c_1}{3}}$ we can find a pair of letters $(a,a') \in P^2$ such that the renormalization  $ \ti{\mathcal{T}}_{{\theta}_0a,{\theta}'_0a'}$ carries $v$ to $(\ul{\theta}a,\ul{\theta}'a',\text{Id},A')$ with $A' \in  X^{\kappa'}_\alpha$ with $\kappa' > \kappa \cdot \lambda$ where $\lambda> 1$ is a constant to be determined. 

Any of these renormalization operators has a very simple visual description when we consider the graphical identification we defined in this section. Precisely, it carries the square that represents $A$ to an inner square $Q_{A(a')}$, that is centered at the point $A(\frac{c_1}{3}a')$ and whose side is equal to $\alpha\frac{c_1}{3}c_0.$ The square that represents the identity is carried into $S(\frac{c_1}{3}a;\frac{c_1}{3}c_0)$.  

We begin by observing that, since $A \in  X^0_\alpha \setminus X^{\kappa_1}_\alpha  $, we can assume, without loss of generality for our next calculation, that $S(0;c_0) \cap A(S(0;c_0)) \subset S_1 \cap A(S_1)$. Dividing then $S_1$ into $3$ squares of side $\frac{c_0}{3}$, $Q_1,\,Q_2,\,Q_3$, we observe that it is impossible for $(Q_i) \cap A(Q_j) \neq \emptyset\, \forall i,j =1,2,3$. On the other hand, the area of 

$$ \bigcup_{a \in P}S(\frac{c_1}{3}a; c_0\cdot \frac{c_1}{3}) \cap \bigcup_{a' \in P} A( S(\frac{c_1}{3}a';c_0 \cdot \frac{c_1}{3}) $$

is at least $c_0^2$ multiplied by $c^2_1(1+|\alpha|^2) - (1-\kappa)(1+|\alpha^2|) $ and it is divided along at most $8$ intersections of the type $S(\frac{c_1}{3}a; c_0\cdot \frac{c_1}{3}) \cap A( S(\frac{c_1}{3}a';c_0 \cdot \frac{c_1}{3}) \neq \emptyset $ for $(a,a') \in P^2$. Since the areas of these squares are $c_0^2$ multiplied by $(\frac{c_1}{3})^2$ and $(\frac{c_1}{3})^2\cdot|\alpha|^2 $  respectively we need only to show that for $c_1$ big enough:

$$ (1+|\alpha|^2) \frac{(c_1^2 - (1-\kappa))}{8} \geq (1+|\alpha|^2) \frac{c_1^2 \cdot \lambda \cdot \kappa }{9} $$

and it is clear that we can choose $c_1$, $\lambda$ and $\kappa$ respecting all the previously fixed constraints in a way that the inequality above is true. Notice that the closer $c_1$ gets to $1$ the closer $\delta$ has to be to $0$. A simple calculation shows that for $\kappa> \kappa_0$ we can take any $c_1$ such that $c_1^2 > \frac{9-9\kappa_0}{9-8\kappa_0}$.

On the other hand, if $A \in X^{\kappa_1}_\alpha  $ with $\kappa_1 \geq \frac{c}{36(1+c)} $ and $\alpha \in R_{c}$, by an analogous argument as above, the intersection

$$ \bigcup_{a \in P} S(\frac{c_1}{3}a; \kappa \frac{c_1}{3} c_0) \cap \bigcup_{a' \in P}A( S(\frac{c_1}{3}a'; \kappa \frac{c_1}{3} c_0))$$

is empty only if, 

$$c_1^2 \cdot \kappa^2 < \left(1-\frac{c}{36(1+c)}\right).$$

On the other hand, if such intersection is non-empty, then the area of the intersection \[S(\frac{c_1}{3}a; \frac{c_1}{3} c_0) \cap A(S(\frac{c_1}{3}a'; \frac{c_1}{3} c_0))\] for some  pair $(a,a') \in P^2$ is larger than $\frac{(1-\kappa)^2}{4} (\frac{c_1^2}{9})(1+|\alpha|^2)$. That way, if $c \geq \sfrac{c_1}{3} > \sfrac{1}{4}$ and $v \in \mathcal{C}$ identified by $(\ul{\theta},\ul{\theta}',\text{Id},A)$ with $A \in X^{\kappa_1}_\alpha $ and $\alpha \in R_{\frac{c_1}{3}}$, we can find a pair of letters $(a,a') \in P^2$ such that the renormalization  $ \ti{\mathcal{T}}_{{\theta}_0a,{\theta}'_0a'}$ carries $v$ to $(\ul{\theta}a,\ul{\theta}'a',\text{Id},A')$ with $A' \in  X^{\kappa_0}_\alpha$, whenever $c_1^2 \geq \frac{179}{180 (1-2\sqrt{\kappa_0})^2  }$.  

We are almost able to construct the recurrent compact set. Before that, fix $\frac{1}{4} < c' \leq \frac{c_1}{3}$, let $\kappa_2 < \frac{c'}{36(1+c')}$ and $\alpha \in R_{c'} \setminus R_{\frac{c_1}{3}}$. We divide into cases:

\begin{enumerate}
    \item $|\alpha|^2> \frac{3}{c_1}$
    
    In this case, if we consider any $v \equiv (\ul{\theta},\ul{\theta}', \text{Id}, A) \in \mathcal{C}$ such that  $A = \alpha \cdot z + \beta \in X^{\kappa_2}_\alpha$ then, by choosing $c_1$ sufficiently close to $1$ and $\kappa_0$ close to $0$ we can find a renormalization operator $ \ti{\mathcal{T}}_{\emptyset,{\theta'}_0a} $ that sends $v$ to $(\ul{\theta}, \ul{\theta}'a',\text{Id},\ti{A'} )$ with $A'=\alpha' \cdot z + \beta', \,\alpha' \in R_{\frac{c_1}{3}},\,\beta' \in \C $ and $A' \in X^{\kappa_0}_{\alpha'}$. Checking the formula for the renormalization operator we have that $\alpha' = \alpha \cdot \frac{c_1}{3} \in R_{\frac{c_1}{3}}$ by definition, so the hard part is to control the translation part of $A'$.
    
    Now, we know that the area of $S(0;c_0)\cap A(S(0;c_0))$ is at most $c_0^2(1+|\alpha|^2)(1-\kappa_2)$. This implies that, for $c_1$ really close to one the area of $S(0;c_0) \cap \bigcup_{a \in P} A(S(\frac{c_1}{3}a;c_0 \cdot \frac{c_1}{3})$ is at least $c_0^2(1+c_1^2|\alpha^2| - (1+|\alpha|^2)(1-\kappa_2)) $. By pigeonhole principle we want to find a choice of $c_1$ and $\kappa_0$ that satisfies:
    
    $$ c_0^2\frac{(1+c_1^2|\alpha^2| - (1+|\alpha|^2)(1-\kappa_2))}{9} \geq c_0^2\cdot \kappa_0 \left(1+\left(\frac{c_1}{3}\right)^2|\alpha|^2\right) .$$
    
    But after some manipulation and applications of the hypothesis $|\alpha|^2> \frac{3}{c_1}$ it is enough to guarantee that:
    
    $$c_1^2 \geq  1+ \frac{12\kappa_0 -4\kappa_2}{3}.$$

    Which is always possible if $\kappa_0$ is sufficiently small and $c_1$ close to $1$.
    
    \item  $|\alpha|^2 < \frac{c_1}{3}$
    
    This case is very similar to the previous one; the difference is that for  $v \equiv (\ul{\theta},\ul{\theta}',\text{Id}, A) \in \mathcal{C} $ such that  $A=\alpha\cdot z+\beta\in X^{\kappa_2}_\alpha$, $a \in P$, we find a renormalization operator $\ti{\mathcal{T}}_{{\theta}_0a,\emptyset}$ that sends $v$ to $(\ul{\theta}a, \ul{\theta}',\text{Id},\ti{A'})$ with $A'=\alpha' \cdot z + \beta', \,\alpha' \in R_{\frac{c_1}{3}},\,\beta' \in \C$ and $A'\in X^{\kappa_0}_{\alpha'}$. Once again, $\alpha'=\alpha \cdot \frac{3}{c_1} \in R_{\frac{c_1}{3}}$, so we proceed to check the translation part.
    
    Again, the area of $S(0;c_0)\cap A(S(0;c_0))$ is at most $c_0^2(1+|\alpha|^2)(1-\kappa_2)$. By definition $A'=(F^{\ul{\theta}a})^{-1}\circ A$, but since $F^{\ul{\theta}a}$ is affine, $A' \in X^{\kappa_0}_{\alpha'}$ if, and only if, the area of $(F^{\ul{\theta}a}(S(0;c_0)  = S(\frac{c_1}{3}a;\cdot c_0\cdot \frac{c_1}{3}) )\cap A(S(0;c_0)$ is larger than
    
    $$\displaystyle{c_0^2 \cdot  \kappa_0 \left((\frac{c_1}{3})^2+|\alpha|^2\right)}.$$
    
    By the same logic as the previous item we are left with the inequality
    
    $$c_0^2\frac{(c_1^2+|\alpha|^2)(1-\kappa_2)}{9} \geq c_0^2 \cdot \kappa_0^2 \left((\frac{c_1}{3})^2 + |\alpha|^2\right) .$$
    
    After some manipulation and applications of the hypothesis $|\alpha|^2 < \frac{c_1}{3}$ it is enough to guarantee that:
    
    $$c_1^2 \geq  1+ \frac{12\kappa_0 -4\kappa_2}{3}.$$
    
    It is no surprise that the quota would be the same given the geometry of the problem.

    \end{enumerate}

Thusly, we can construct a \textit{ recurrent compact set} $\mathcal{L} \subset \mathcal{C}$ (which we draw in the final page) as a union $\mathcal{L} = \mathcal{L}^{-1} \cup \mathcal{L}^{0} \cup \mathcal{L}^{1}$, $\mathcal{L}^i = {P^{\N}}^- \times {P^{\N}}^- \times \text{Id} \times L^i, \text{ for } i=-1,\,0,\,1$, where:

\begin{itemize}
    \item $\displaystyle{{L}^1 = \bigcup_{\alpha \in R_{c'}^1} X^{\kappa_2}_{\alpha}} $; $\displaystyle{R_{c'}^1 = \left\{ \alpha \in R_{c'}, \, |\alpha| > \sqrt{\frac{3}{c_1}}\right\} }$
    \item $\displaystyle{{L}^{-1} = \bigcup_{\alpha \in R_{c'}^{-1} } X^{\kappa_2}_\alpha}$; $\displaystyle{ R_{c'}^{-1} = \left\{\alpha \in R_{c'}, \, |\alpha| < \sqrt{\frac{c_1}{3}}\right\}} $
    \item $\displaystyle{{L}^0 = \bigcup_{\alpha \in R_{c'} \setminus (R_{c'}^1 \cup R_{c'}^{-1} )} X^{\kappa_0}_\alpha,                           }$
\end{itemize}

and $\kappa_0,\, \kappa_2, \text{ and } c_1$ are chosen respecting the constraints we have already fixed. As we have already shown, for almost all $v \in \mathcal{L}$ one of the renormalization operators $T$ we already found above makes $T(v) \in \text{int}(\mathcal{L})$. The exceptions are the $v=(\ul{\theta},\ul{\theta'}, \text{Id}, A)$ where $A=\alpha \cdot z + \beta$  with $|\alpha|^2= \frac{c_1}{3} \text{ or } \frac{c_1}{3} $ and $A \in X^{\kappa_0}_\alpha \setminus A \in X^{\kappa_2}_\alpha$. Yet, in this case we can repetitively apply the renormalization operators previously described appropriate to this case to obtain a sequence $v_n = T_{\ul{\theta}^na_n,{\ul{\theta}'}^na'_n} (v_{n-1})$ for which 

\begin{itemize}
    \item $v_0=v$ 
    \item $({\ul{\theta}^n,{\ul{\theta}'}^n}) = ({\ul{\theta}^{n-1}a_{n-1},{\ul{\theta}'}^{n-1}{a'}_{n-1}})$ and
    \item $v_n\in X^{\lambda^n\kappa_0}_\alpha$
\end{itemize}

hence, if $n$ is large enough ${\lambda^n\kappa_0} > \kappa_2 \implies v_n \in \text{int}(\mathcal{L})$ as we wished to obtain.

\end{proof}

\begin{remark}
Given the constraints on the proof above, more importantly $\kappa_2 < \frac{c'}{36(1+c')} $, we can calculate that for \[c_1^2 \geq \displaystyle{\min_{\kappa_0 \in [0,1]} \max\bigg\{\frac{9-9\kappa_0}{9-8\kappa_0}, 1+\frac{12\kappa_0-\frac{1}{90}}{3}, \frac{179}{180 (1-2\sqrt{\kappa_0})^2   }\bigg\}} = x^2\] 
$\kappa_0$ can be chosen such that the construction above works at all steps. A simple calculation (it can be done by hand, using $\kappa_0=10^{-6}$) show that $x^2 <1-\sfrac{(10^-6)}{9}$ and so, since by definition $c_1 < \frac{3c_0}{2+c_0}$ we can estimate that for all $\delta < 7 \cdot 10^{-8}$ the Cantor set defined in this section has a stable intersection with itself (provided we choose $c_1 \text{ sufficiently close to } \frac{3c_0}{2+c_0}$). This quota is not optimal and may be greatly improved by adaptations in the argument, since a lot of area is \say{wasted} in the estimatives, especially in the part regarding $\frac{179}{180 (1-2\sqrt{\kappa_0})^2}$.
\end{remark}

\begin{figure}[H]
\includegraphics[scale=1]{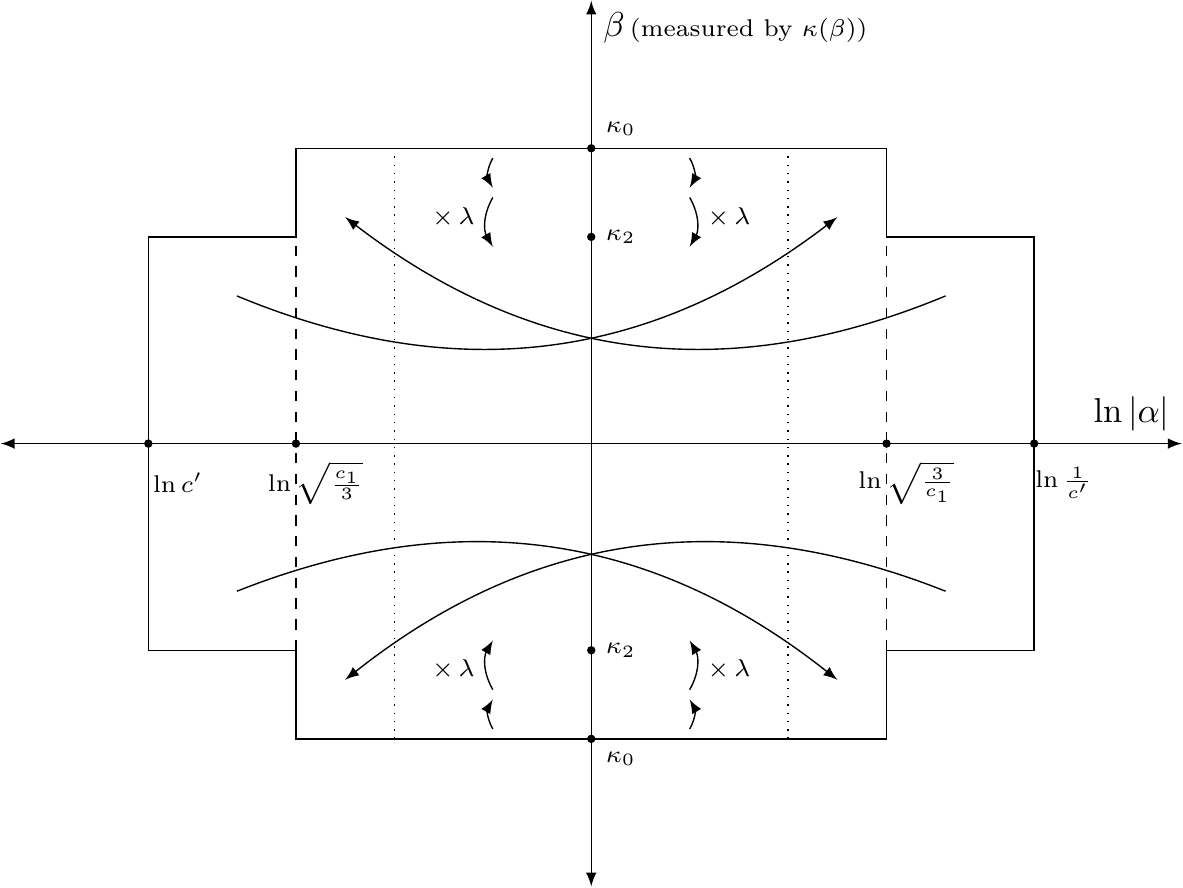}
\caption{A diagram representing the recurrent compact set. The coordinate $\beta$ is measured by $\kappa(\beta)$, the proportion of area of intersection between the original square and its image by $\alpha\cdot z + \beta$. This way, being closer to the axis means that this area is close to the maximal proportion, whereas being far means it is close to zero. The arrows indicate the action of the renormalizations considered above. Note that [Id, Id] is represented by the origin.}
 \label{fig:Test}
\end{figure}

\bibliographystyle{abbrv}
\bibliography{refs}

\end{document}